\documentclass[11pt,reqno]{amsart}

\usepackage[T1]{fontenc}
\usepackage{lmodern}
\usepackage{mathrsfs}
\usepackage{dsfont}
\usepackage{enumerate}
\usepackage[a4paper,margin=1in]{geometry}
\usepackage[colorlinks=true,linkcolor=blue,citecolor=red,urlcolor=blue]{hyperref}

\allowdisplaybreaks

\theoremstyle{plain}
\newtheorem{theorem}{Theorem}[section]
\newtheorem{lemma}[theorem]{Lemma}
\newtheorem{proposition}[theorem]{Proposition}
\newtheorem{corollary}[theorem]{Corollary}
\theoremstyle{definition}
\newtheorem{question}{Question}

\newcommand{\RR}{\mathbb{R}}
\newcommand{\ZZ}{\mathbb{Z}}
\newcommand{\NN}{\mathbb{N}}
\newcommand{\EE}{\mathbb{E}}

\newcommand{\PP}{\mathbb{P}}
\newcommand{\FF}{\mathcal{F}}
\newcommand{\cH}{\mathcal{H}}
\newcommand{\F}{\mathscr{F}}
\newcommand{\1}{\mathds{1}}
\newcommand{\Osc}{\operatorname{Osc}}

\begin{document}

\author[Shuo Qin]{Shuo Qin}
\address[Shuo Qin]{Beijing Institute of Mathematical Sciences and Applications, and Yau Mathematical Sciences Center, Tsinghua University}
\email{qinshuo@bimsa.cn}

\title{Cover times and ranges of elephant random walks}
\date{}

\begin{abstract}
We study cover times on discrete tori and ranges on $\ZZ^d$ for elephant random walks with memory parameter $p\in[0,1)$. In dimension one there is a phase transition at $p=3/4$. Away from criticality we determine the exact first-order asymptotics of the mean cover time, while at $p=3/4$ the mean is of order $L^2/\sqrt{\log L}$, a factor $\sqrt{\log L}$ larger than the natural fluctuation scale $L^2/\log L$. We identify the limiting distribution under each of the three natural fluctuation normalizations. In dimensions $d\geq2$, for every fixed $p<1$, the cover time has the same order as for simple random walk, in expectation and with high probability. The leading constants also agree when $p<(2d+1)/(4d)$, and at equality when $d\geq3$. For the range on $\ZZ^d$, we prove $L^1$ convergence to the simple random walk asymptotics when $d=2$ and $p<5/8$, and for every $p<1$ when $d\geq3$. Finally, we give a cover-time upper bound for generalized step-reinforced random walks on finite groups in terms of a conditional $L^\infty$ mixing profile.
\end{abstract}

\maketitle

\section{Introduction}
\label{secintro}

The elephant random walk (ERW) on $\ZZ$ was introduced by Sch\"utz and Trimper \cite{schutz2004elephants} as a simple non-Markovian random walk with long memory. Its extension to $\ZZ^d$ was studied by Bercu and Laulin \cite{MR3962977}. The walk starts at the origin. Its first increment is chosen uniformly from the $2d$ nearest-neighbor directions. At each subsequent time, the elephant recalls one of its previous increments uniformly at random. With probability $p\in[0,1]$, it repeats that increment, and with probability $1-p$ it chooses uniformly from the other $2d-1$ directions.

Write $\widetilde S$ for this walk on $\ZZ^d$. For an integer $L\geq2$, the ERW on the discrete torus
\[
    \ZZ_L^d:=(\ZZ/L\ZZ)^d
\]
is its canonical projection $S_n=\widetilde S_n\pmod L$. We study its cover time
\[
    \tau_{\rm cov}:=\inf\{n\geq0:\{S_0,S_1,\ldots,S_n\}=\ZZ_L^d\}.
\]
The case $p=1$ is degenerate: all increments equal the first one, so $\tau_{\rm cov}=L-1$ for $d=1$ and the walk never covers $\ZZ_L^d$ for $d\geq2$. We therefore assume $p<1$ throughout. It is often more convenient to use the reinforcement parameter 
\begin{equation}
    \label{defalpha}
    \alpha:=\frac{2dp-1}{2d-1}.
\end{equation}
In particular, $\alpha=2p-1$ in dimension one.

The cover time behaves quite differently in dimension one and in higher dimensions.
In dimension one, both the mean cover time (Theorem \ref{thmcover}) and its natural fluctuation scale (Proposition \ref{prop1dlimit}) depend on the memory parameter.

\begin{theorem}
\label{thmcover}
Let $S$ be an elephant random walk on $\ZZ_L$ with memory parameter $p\in[0,1)$. There exist constants $C_-(p),C_+(p)\in(0,\infty)$ such that, as $L\to\infty$,
\[
\EE\tau_{\rm cov}\sim\begin{cases}
C_-(p)L^2,&p<3/4,\\
C_+(p)L^{5-4p},&p>3/4,
\end{cases}
\]
whereas
\[
\EE\tau_{\rm cov}=\Theta\left(\frac{L^2}{\sqrt{\log L}}\right),\qquad p=3/4.
\]
\end{theorem}

For any function $f$ in the Skorokhod space $D([0,\infty))$ and $t\geq0$, write
\[
\Osc_t(f):=\sup_{0\leq s\leq t}f(s)-\inf_{0\leq s\leq t}f(s),
\]
and, for $\lambda>0$, set
\begin{equation}
\label{defPhi}
\Phi(\lambda,f):=\inf\{t\geq0:\Osc_t(f)\geq\lambda\}.
\end{equation}

\begin{proposition}
\label{prop1dlimit}
Let $S$ be as in Theorem~\ref{thmcover}, and let $(B_t)_{t\geq0}$ be a standard Brownian motion.
\begin{enumerate}[(i)]
\item If $p<3/4$, then as $L\to\infty$,
\[
\frac{\tau_{\rm cov}}{L^2}\Longrightarrow\Phi(1,\widehat B),
\]
where the noise-reinforced Brownian motion $\widehat B$ has the same law as
\[
\left(\frac{t^{2p-1}}{\sqrt{3-4p}}B_{t^{3-4p}}\right)_{t\geq0},
\]
with value $0$ at $t=0$ by continuous extension.
\item If $p=3/4$, then as $L\to\infty$,
\[
\frac{\tau_{\rm cov}\log L}{L^2}\Longrightarrow\frac{1}{2B_1^2}.
\]
\item If $p>3/4$, let
\[
Y:=\lim_{n\to\infty}\frac{\widetilde S_n}{n^{2p-1}},
\]
which exists and is non-zero almost surely. If the walks for different $L$ are obtained by projecting the same lifted walk, then
\[
\lim_{L\to\infty}\frac{\tau_{\rm cov}}{L^{1/(2p-1)}}=\frac{1}{|Y|^{1/(2p-1)}}
\qquad\text{almost surely}.
\]
\end{enumerate}
\end{proposition}

For $p<3/4$, we will show in (\ref{subcritmeanconstant}) that the constant $C_-(p)$ in Theorem~\ref{thmcover} is
\begin{equation}
\label{subcritmeanconstantintro}
C_-(p)=\EE\Phi(1,\widehat B).
\end{equation}
An integral expression for $C_+(p)$ is given in Proposition~\ref{thmsuperprofiles}.
At and above the critical parameter, the limiting random variables in Proposition~\ref{prop1dlimit} have infinite mean.
In the superdiffusive case, this follows from the positivity of the density of $Y$ at the origin \cite[Theorem 1.3]{MR5079601}.
The fluctuation limits therefore do not determine the first-order asymptotics of the mean cover time.
Its larger order comes from rare paths that remain inside an interval of diameter comparable to $L$ for times of order $L^2$.
We estimate the probabilities of such paths in Section~\ref{secexitd1}.

For $d\geq2$, the cover time has the same order as for simple random walk throughout $p\in[0,1)$. The leading constant is unchanged in the diffusive regime, and also at criticality when $d\geq3$. For fixed $d\geq2$, put
\begin{equation}
    \label{defHdLtLstar}
    H(d,L):=\begin{cases}
L^2\log^2L,&d=2,\\
L^d\log L,&d\geq3,
\end{cases}
\qquad
t_L^\star:=\begin{cases}
\displaystyle\frac4\pi L^2\log^2L,&d=2,\\[1mm]
\displaystyle g_d(0)L^d\log(L^d),&d\geq3,
\end{cases}
\end{equation}
where $g_d(0)$ is the Green function at the origin for simple random walk on $\ZZ^d$, $d\geq3$.

\begin{theorem}
\label{thmhighd}
Let $S$ be an elephant random walk on $\ZZ_L^d$ with fixed $d\geq2$ and $p\in[0,1)$.
\begin{enumerate}[(i)]
\item If $\alpha<1/2$, or if $d\geq3$ and $\alpha=1/2$, then
\begin{equation}
\label{sharpcoverLone}
\lim_{L\to\infty}\frac{\tau_{\rm cov}}{t_L^\star}=1\qquad\text{in }L^1.
\end{equation}
\item For every $p\in[0,1)$, $\EE\tau_{\rm cov}=\Theta(H(d,L))$. Moreover, there exist positive constants $\kappa(p,d)$, $\widetilde\kappa(p,d)$ and $\delta(p,d)$ such that
\begin{equation}
\label{tauconcend23}
\sup_{L\geq2}L^{\delta(p,d)}\PP\left(\tau_{\rm cov}\leq\kappa(p,d)H(d,L)\text{ or }\tau_{\rm cov}\geq\widetilde\kappa(p,d)H(d,L)\right)<\infty.
\end{equation}
\end{enumerate}
\end{theorem}

We also study the number of distinct sites visited by the walk before projection.
The one-dimensional distributional limits are given in Corollary~\ref{cor1drangelimit}.
Let
\[
\mathcal R_n:=\{\widetilde S_0,\ldots,\widetilde S_n\},
\quad \text{ and } \quad
\gamma_d:=\frac1{g_d(0)},\ \text{ for } d\geq 3.
\]
The constant $\gamma_d$ is the probability that simple random walk on $\ZZ^d$, started at the origin, never returns there after time zero. The following theorem gives the same range asymptotics as for simple random walk.

\begin{theorem}
\label{thmrange}
Let $\widetilde S$ be an elephant random walk on $\ZZ^d$ with fixed memory parameter $p\in[0,1)$.
\begin{enumerate}[(i)]
\item If $d=2$ and $p<5/8$, then
\[
\lim_{n\to\infty}\frac{\log n}{n}|\mathcal R_n|=\pi
\qquad\text{in }L^1.
\]
\item If $d\geq3$, then
\[
\lim_{n\to\infty}\frac{|\mathcal R_n|}{n}=\gamma_d
\qquad\text{in }L^1.
\]
\end{enumerate}
\end{theorem}

In dimensions $d\geq3$, the limiting rate of discovery of new sites is therefore independent of $p$, including in the critical and superdiffusive regimes. 

The higher-dimensional proofs compare the increments on suitable time intervals with those of simple random walk.
Relative entropy estimates allow us to transfer its cover-time and range asymptotics in the diffusive regime.
Bounds on return probabilities and on the upper tail of the cover time control the expectations.
For ranges in dimensions at least three, summable return bounds also allow us to treat the critical and superdiffusive regimes.

Section~\ref{secnotation} collects the preliminary estimates.
Sections~\ref{secexitd1} and \ref{secHighD} prove the cover-time results, and Section~\ref{secrange} treats ranges.
Section~\ref{secfinitegroup} extends the conditional hitting argument to obtain cover-time bounds for generalized step-reinforced random walks on finite groups.

\section{Notation and preliminary estimates}
\label{secnotation}

Throughout, $\NN:=\{0,1,2,\ldots\}$ and $[n]:=\{1,2,\ldots,n\}$ for $n\geq1$. Let $e_1,\ldots,e_d$ be the coordinate vectors of $\ZZ^d$. We retain the direction labels $X_n\in\{\pm e_1,\ldots,\pm e_d\}$ before projecting onto the torus, and write
\[
\widetilde S_n:=\sum_{j=1}^n X_j,\qquad S_n:=\widetilde S_n\pmod L,
\qquad \widetilde S_0=S_0=0.
\]
 Let $u_n$ be uniform on $[n-1]$, independently for $n\geq2$. The first direction $X_1$ is uniform. At time $n\geq2$, the walk repeats $X_{u_n}$ with probability $p$ and otherwise chooses uniformly from the other $2d-1$ directions. We use the filtration
\[
\FF_n:=\sigma(X_1,\ldots,X_n),\qquad n\geq1,
\]
and take $\FF_0$ to be trivial. For $y\in\ZZ_L^d$, let
\[
\tau_y:=\inf\{n\geq0:S_n=y\},\qquad
\tau_{\rm cov}=\max_{y\in\ZZ_L^d}\tau_y.
\]
The symbols $C(a_1,\ldots,a_k)$ and $c(a_1,\ldots,a_k)$ denote positive constants depending only on the displayed parameters. Their values may change from line to line. We use $\Longrightarrow$ for convergence in distribution.

For each direction $e\in\{\pm e_1,\ldots,\pm e_d\}$, define
\begin{equation}
\label{defNni}
N_n(e):=\sum_{j=1}^n\1_{\{X_j=e\}},\qquad
\Delta_n(e):=N_n(e)-\frac{n}{2d}.
\end{equation}
Let $\Delta_n$ be the $2d$-dimensional vector $(\Delta_n(e))$ with Euclidean norm $\|\Delta_n\|$. By the definition of the walk,
\begin{equation}
\label{PXn1}
\begin{aligned}
\PP(X_{n+1}=e\mid\FF_n)
&=p\frac{N_n(e)}n+\frac{1-p}{2d-1}\left(1-\frac{N_n(e)}n\right)\\
&=\frac1{2d}+\frac{\alpha\Delta_n(e)}n,
\end{aligned}
\end{equation}

When it is useful to display the parameter, we write $\PP^{(\alpha)}$ for the law of the direction sequence and the resulting walk. Thus $\PP^{(0)}$ denotes simple random walk, corresponding to $p=1/(2d)$ rather than $p=0$. 

We shall use the following estimates from \cite[Lemma 2.5]{peres2026transition}. Set
\[
V_\alpha(n):=\begin{cases}
n,&\alpha<1/2,\\
n\log(n+1),&\alpha=1/2,\\
n^{2\alpha},&\alpha>1/2.
\end{cases}
\]
\begin{lemma}
\label{empiconvunifprop}
For fixed $p\in[0,1)$ and $d\geq1$, there exist $c,C>0$ such that, for every $n\geq1$ and $v\geq0$,
\begin{equation}
\label{directioncounttail}
\PP(\|\Delta_n\|\geq v)\leq C\exp\left(-\frac{cv^2}{V_\alpha(n)}\right),
\end{equation}
and
\begin{equation}
\label{directioncountmoment}
\EE\|\Delta_n\|^2\leq C V_\alpha(n).
\end{equation}
\end{lemma}

The cited lemma applies to the $2d$ direction labels, independently of their projection onto the torus.

When $\alpha\geq0$, \eqref{PXn1} gives the usual step-reinforced construction of elephant random walk. Independently of the variables $(u_n)$, take i.i.d. Bernoulli variables $(\xi_n)_{n\geq2}$ with parameter $\alpha$ and i.i.d. uniform directions $(g_n)_{n\geq1}$. Set $X_1=g_1$ and, for $n\geq2$, set $X_n=X_{u_n}$ if $\xi_n=1$ and $X_n=g_n$ otherwise. The associated forest $\F_n$ has an edge from $j$ to $u_j$ exactly when $\xi_j=1$. Write $\mathscr I_n$ for its isolated vertices.

We recall the conditional forest estimates from \cite[Lemma 2.2 and Proposition 2.3]{peres2026transition}. For a finite $\FF$-stopping time $\tau$ and $n\geq1$, put
\[
\mathscr I(\tau,n):=\mathscr I_{\tau+n}\cap\{\tau+1,\ldots,\tau+n\}.
\]
Condition on $\FF_\tau$, the forest up to time $\tau+n$, and all labels $g_j$ outside $\mathscr I(\tau,n)$ up to that time. Denote this information by $\cH_{\tau,n}$. There is a $\cH_{\tau,n}$-measurable remainder $R_{\tau,n}$ such that, on the torus,
\begin{equation}
\label{eqconditionalforestdecomp}
S_{\tau+n}-S_\tau
=\sum_{j\in\mathscr I(\tau,n)}g_j+R_{\tau,n}.
\end{equation}
Conditionally on $\cH_{\tau,n}$, the labels in this sum remain independent uniform directions. None of them occurs in the observed past or in another component of the forest. Moreover,
\begin{equation}
\label{eqestisover}
\PP\left(\left.|\mathscr I(\tau,n)|\leq\frac{(1-\alpha)n}{8}\,\right|\FF_\tau\right)
\leq2\exp\left(-\frac{(1-\alpha)n}{18}\right).
\end{equation}

We finish with an elementary observation that will be used in Sections~\ref{secHighD} and \ref{secfinitegroup}.
\begin{lemma}
\label{lemblockcover}
Let $Z$ be adapted to a filtration $(\mathcal A_m)_{m\geq0}$ and take values in a set $V$ with $N\geq2$ elements. Write $\tau_y^Z:=\inf\{m\geq0:Z_m=y\}$ and $\tau_{\rm cov}^Z:=\max_{y\in V}\tau_y^Z$. Suppose that an integer $K\geq1$ and $\rho\in(0,1]$ satisfy
\[
\PP\bigl(y\in\{Z_m,\ldots,Z_{m+K}\}\mid\mathcal A_m\bigr)\geq\rho
\]
almost surely for every deterministic $m\geq0$ and $y\in V$. Then, for every integer $j\geq0$,
\[
\PP(\tau_y^Z>jK)\leq(1-\rho)^j,\qquad
\PP(\tau_{\rm cov}^Z>jK)\leq\min\{1,N(1-\rho)^j\}.
\]
In particular, for an absolute constant $C$,
\[
\EE\tau_y^Z\leq\frac K\rho,\qquad
\EE\tau_{\rm cov}^Z\leq\frac{CK\log N}{\rho},
\]
and, for every $r\geq0$,
\[
\PP\left(\tau_{\rm cov}^Z>\frac{CK}{\rho}(\log N+r)\right)\leq e^{-r}.
\]
\end{lemma}
\begin{proof}
On $\{\tau_y^Z>jK\}$, the conditional probability of avoiding $y$ for one more block is at most $1-\rho$. Iteration gives the first tail bound, and a union bound over $y$ gives the second. To sum the latter bound, use $(1-\rho)^j\leq e^{-\rho j}$ and
\begin{equation}
    \label{eqsummin1Nerho}
    \sum_{j\geq0}\min\{1,Ne^{-\rho j}\}
\leq1+\int_0^\infty\min\{1,Ne^{-\rho t}\}\,dt
=1+\frac{\log N+1}{\rho}.
\end{equation}
The expectation bounds follow by grouping the tail sum into blocks of length $K$. Taking $j=\lceil(\log N+r)/\rho\rceil$ proves the last assertion, with an absolute constant absorbing the rounding.
\end{proof}

\section{Cover times in dimension one}
\label{secexitd1}

In this section $d=1$, recall that $\alpha=2p-1$ and $\widetilde S$ denotes the lift defined in Section~\ref{secnotation}. The filtration $\FF_n=\sigma(X_1,\ldots,X_n)$ also equals $\sigma(\widetilde S_0,\ldots,\widetilde S_n)$. Every nearest-neighbor path visits all the integers between its running minimum and maximum. Hence
\begin{equation}
\label{coverrangeidentity}
\tau_{\rm cov}=\inf\left\{k\geq1:\max_{0\leq j\leq k}\widetilde S_j-\min_{0\leq j\leq k}\widetilde S_j\geq L-1\right\}.
\end{equation}
For a positive integer $m$, let
\begin{equation}
\label{defexittime}
\sigma_m:=\inf\{n\geq0:\widetilde S_n\notin(-m,m)\}
=\inf\{n\geq0:|\widetilde S_n|=m\}
\end{equation}
be the exit time from $(-m,m)$. The range identity implies
\begin{equation}
\label{equcovexit1D}
\sigma_{\lfloor L/2\rfloor}\leq\tau_{\rm cov}\leq\sigma_L.
\end{equation}
Andr\'e and Zuazn\'abar \cite{Andre2026} proved that, for $p<3/4$, $\EE\sigma_m\sim C(p)m^2$ for some positive constant $C(p)$. The endpoint $p=0$, although omitted from the statements of \cite[Proposition 1, Theorem 1]{Andre2026}, is covered by the same proof. We shall use their exponential tail estimate below and obtain the exact asymptotic constant for the cover time directly from the functional limit theorem.

For $p\geq3/4$, the corresponding exit-time result is as follows. In the superdiffusive case the constant is identified in \eqref{superexitmean}.

\begin{proposition}
\label{propcover1d}
Let $\widetilde S$ be an elephant random walk on $\ZZ$ with memory parameter $p\in[3/4,1)$, and let $\sigma_m$ be as in \eqref{defexittime}. The following estimates hold as $m\to\infty$.
\begin{enumerate}[(i)]
\item If $p=3/4$, then
\[
\EE\sigma_m=\Theta\left(\frac{m^2}{\sqrt{\log m}}\right).
\]
\item If $p>3/4$, then there exists $C_{\rm exit}(p)\in(0,\infty)$ such that
\[
\EE\sigma_m\sim C_{\rm exit}(p)m^{5-4p}.
\]
\end{enumerate}
\end{proposition}

\subsection{Tail bounds and moment transfer}
\label{sectailtransfer1d}

By coupling $\PP^{(\alpha)}$ ($\alpha \geq 0$) and $\PP^{(0)}$, we can prove the following result. 

\begin{lemma}
\label{tailm2up1derwlem}
  There exists a constant $\rho \in (0,1)$ such that for any $ \alpha \in [0,1)$, and any $m, k \geq 1$,
    \[
  \PP^{(\alpha)}(\sigma_m > m^2 + k m^2) \leq \rho^k \PP^{(\alpha)}(\sigma_m > m^2).
    \]
    In particular, 
    \[
\sum_{n > m^2} \PP^{(\alpha)}(\sigma_m >n) \leq \frac{m^2}{1-\rho} \PP^{(\alpha)}(\sigma_m > m^2).
    \]
\end{lemma}
\begin{proof}
The coupling in \cite[Proposition 2.5]{MR4908097}, stated there for $j=0$, also works after any deterministic time $j$. Indeed, when $|\widetilde S_n|=r>0$, the conditional probability of increasing the absolute value is
\[
\frac12+\frac{\alpha r}{2n}\geq\frac12.
\]
At zero, the absolute value increases to $1$. Couple the two absolute-value processes using the same uniform variables. Their parities agree, so unequal values differ by at least $2$. When they are equal, the preceding probability comparison preserves their order. Thus, given $\FF_j$, we can couple the future walk with a simple random walk $\widehat S$ started at $\widehat S_j=\widetilde S_j$ so that $|\widehat S_n|\leq|\widetilde S_n|$ for all $n\geq j$. Therefore,
\[
\PP^{(\alpha)}( \sigma_m > j+m^2 \mid \FF_j) \1_{\{\sigma_m > j\}} \leq \rho \1_{\{\sigma_m > j\}},
\]
where
\[
\rho:=\sup_{m\geq2}\sup_{|x|<m}\PP_x^{(0)}(\sigma_m>m^2)<1.
\]
To see that $\rho<1$, for $x\geq0$ the probability of exiting by time $m^2$ is at least $\PP_0^{(0)}(\widehat S_{m^2}\geq m)$; for $x<0$, use symmetry. By the central limit theorem, these probabilities have a positive limit as $m\to\infty$.  The first assertion then follows from iteration. Thus,
\[
\sum_{n > m^2} \PP^{(\alpha)}(\sigma_m >n) \leq \sum_{k=0}^{\infty} m^2 \PP^{(\alpha)}(\sigma_m >m^2 +k m^2 ) \leq m^2 \PP^{(\alpha)}(\sigma_m >m^2 ) \sum_{k=0}^{\infty}  \rho^k, 
\]
which proves the second inequality.
\end{proof}

When $\alpha=1/2$, \cite[Theorem 3]{MR4377583} implies that there exists a positive constant $C$ such that for all $n > 1$ and all $y \in \ZZ$ with $|y|\leq \sqrt{n} (\log n)^{1/3} (=o(\sqrt{n \log n}))$, one has

\begin{lemma}\label{lemcriticalpointbound}
Let $\widetilde S$ be a one-dimensional ERW with memory parameter
$p=3/4$. There exists a constant $C>0$ such that, for every $n\ge2$,
\begin{equation}\label{eqcriticalpointbound}
    \sup_{y\in\mathbb Z}\mathbb P(\widetilde S_n=y)
    \le \frac{C}{\sqrt{n\log n}}.
\end{equation}
\end{lemma}

\begin{proof}
Write $\mathbb P_+$ and $\mathbb E_+$ for probability and expectation conditional on $X_1=1$. Note that the conditional law given $X_1=-1$ is the reflection of the law
given $X_1=1$. By averaging these two conditional laws, we see that it suffices to show that
\begin{equation}
    \label{eqripointbdplus}
      \sup_{y\in\mathbb Z}\mathbb P_+(\widetilde S_n=y)
    \le\frac{C}{\sqrt{n\log n}}.
\end{equation}
By
\cite[Proposition~2.2]{MR4926011}, the distribution of
\[
    Z_n:=\frac{\widetilde S_n+n}{2}
\]
under $\mathbb P_+$ is log-concave on its support $\{1,\ldots,n\}$. We first give an elementary bound for its largest point mass. Fix $n\ge2$, and write
\[
    q_j:=\mathbb P_+(Z_n=j),\quad j \in [n]; \qquad
    a:=\max_j q_j=q_m,\qquad
    Q_j:=\sum_{r\ge j}q_r.
\]
Log-concavity means that $q_{j+1}/q_j$ is non-increasing on the
support. Hence, for $j\ge m$ with $q_j>0$,
\[
    \frac{Q_j}{q_j}
    =\sum_{r\ge0}\frac{q_{j+r}}{q_j} = \sum_{r\ge0}\frac{q_{j+r}}{q_{j+r-1}}\cdot\dots\cdot\frac{q_{j+1}}{q_j}
    \le \sum_{r\ge0}\frac{q_{m+r}}{q_m}
    =\frac{Q_m}{a}
    \le \frac1a.
\]
It follows that $Q_{j+1}=Q_j-q_j\le(1-a)Q_j$.
Iteration, together with the same argument for the left tail, gives
\[
    \mathbb P_+(|Z_n-m|\ge k)\le2(1-a)^k,
    \qquad k\ge1.
\]
Consequently,
\[
\begin{aligned}
     \operatorname{Var}_+(Z_n)
    &\leq \mathbb E_+(Z_n-m)^2  =\sum_{k\ge1}(2k-1)\mathbb P_+(|Z_n-m|\ge k)  \\
    &\leq 2\sum_{k\ge1}(2k-1)(1-a)^k
    =\frac{2(1-a)(2-a)}{a^2} \le\frac4{a^2}.
\end{aligned}
\]

By symmetry and \cite{schutz2004elephants}, as $n \to \infty$,
\[
    \operatorname{Var}_+(Z_n)
    =\frac14\operatorname{Var}_+(\widetilde S_n) = \frac14\operatorname{Var}(\widetilde S_n)
    \sim\frac14 n\log n.
\]
Combining this with the preceding bound for $a$, and increasing
$C$ to cover the finitely many small values of $n$, proves \eqref{eqripointbdplus}.
\end{proof}

For $\alpha > 1/2$, we shall use the following estimate on the transition probabilities.

\begin{lemma}
\label{anticoncenERWlara12}
Let $\widetilde S$ be an ERW with reinforcement parameter $\alpha\in[0,1)$. There exists $C(\alpha)>0$ such that, for all $n\geq1$,
\[
\sup_{y\in\ZZ}\PP^{(\alpha)}(\widetilde S_n=y)\leq C(\alpha)n^{-\alpha}.
\]
\end{lemma}
\begin{proof}
Write $m_n:=\sup_y\PP(\widetilde S_n=y)$. Since $\widetilde S_n=2\Delta_n(e_1)$, \eqref{PXn1} gives
\begin{equation}
\label{condtranStild}
\PP(\widetilde S_{n+1}=y\pm1\mid\widetilde S_n=y)
=\frac12\left(1\pm\frac{\alpha y}{n}\right)
\end{equation}
for every reachable $y$ at time $n$. At an interior site $|y|\leq n-1$ of the support at time $n+1$,
\[
\begin{aligned}
\PP(\widetilde S_{n+1}=y)
={}&\frac12\left(1+\frac{\alpha(y-1)}n\right)\PP(\widetilde S_n=y-1)\\
&+\frac12\left(1-\frac{\alpha(y+1)}n\right)\PP(\widetilde S_n=y+1)
\leq\left(1-\frac\alpha n\right)m_n.
\end{aligned}
\]
Here we used that both coefficients are nonnegative. The two boundary probabilities are $\PP(\widetilde S_{n+1}=\pm(n+1))=p^n/2$, where $p=(1+\alpha)/2<1$ is the memory parameter. Consequently,
\[
m_{n+1}\leq\max\left\{\left(1-\frac\alpha n\right)m_n,\frac{p^n}{2}\right\}.
\]
Choose $C(\alpha)\geq\sup_{n\geq1}n^\alpha p^{n-1}/2$, which is finite since $p<1$. One can then prove $m_n\leq C(\alpha)n^{-\alpha}$ by induction, using the inequality $(1+1/n)^{-\alpha}\geq 1-\alpha/n$.
\end{proof}

For a fixed memory parameter $p\in[0,1)$, define
\begin{equation}
\label{defbLp}
b_p(L):=\begin{cases}
1,&p<3/4,\\
\sqrt{\log L},&p=3/4,\\
L^{4p-3},&p>3/4,
\end{cases}
\end{equation}
and, for $t\geq0$,
\begin{equation}
\label{deftailprofiles}
F_{L,p}^{\rm cov}(t):=b_p(L)\PP(\tau_{\rm cov}>\lfloor L^2t\rfloor),\qquad
F_{L,p}^{\rm exit}(t):=b_p(L)\PP(\sigma_L>\lfloor L^2t\rfloor).
\end{equation}
The next proposition gives an integrable bound for these tail functions, allowing us to pass from their limits to exact mean asymptotics.

\begin{proposition}
\label{proptailUI1d}
For every fixed $p\in[0,1)$, there exists a nonnegative function $h_p\in L^1(0,\infty)$ such that
\[
F_{L,p}^{\bullet}(t)\leq h_p(t),
\qquad L\geq2,\quad t>0,\quad
\bullet\in\{\mathrm{cov},\mathrm{exit}\}.
\]
In particular, both families in \eqref{deftailprofiles} are uniformly integrable with respect to Lebesgue measure on $(0,\infty)$.
\end{proposition}

\begin{proof}
Since $\tau_{\rm cov}\leq\sigma_L$, it is enough to consider $F_{L,p}^{\rm exit}$. We give a common integrable upper bound. If $p<3/4$, the exponential estimate of \cite[Theorem 1]{Andre2026} gives
\[
\PP(\sigma_L>n)\leq C(p)\exp\left(-\frac{c(p)n}{L^2}\right).
\]
The same argument covers $p=0$ since the bound in Section 3.1(ii) of that paper uses only $p< 1/2$. This gives the required bound in the subcritical regime.

Suppose $p=3/4$ and put $u_L:=(\log L)^{-2/3}$. If $u_L\leq t\leq1$ and $n=\lfloor L^2t\rfloor$, then, for all large $L$, we have $L\leq\sqrt n(\log n)^{1/3}$. Thus \eqref{eqcriticalpointbound} and a union bound over $y \in [-L,L]$ yield
\[
F_{L,p}^{\rm exit}(t)
\leq\sqrt{\log L}\,\PP(|\widetilde S_n|<L)
\leq Ct^{-1/2}\leq Ct^{-3/4}.
\]
For $0<t<u_L$, the trivial bound gives $F_{L,p}^{\rm exit}(t)\leq\sqrt{\log L}\leq t^{-3/4}$. Consequently, after absorbing finitely many small $L$ into the constant, one has
\[
F_{L,3/4}^{\rm exit}(t)\leq Ct^{-3/4},\qquad0<t\leq1.
\]

For $p>3/4$, write $\alpha=2p-1$ and $\delta=2\alpha-1$. Lemma~\ref{anticoncenERWlara12} gives, for $L^{-2}\leq t\leq1$,
\[
F_{L,p}^{\rm exit}(t)
\leq L^\delta\PP(|\widetilde S_{\lfloor L^2t\rfloor}|<L)
\leq C(p)t^{-\alpha}.
\]
For $0<t<L^{-2}$,
\[
F_{L,p}^{\rm exit}(t)\leq L^\delta\leq t^{-\delta/2}\leq t^{-\alpha}.
\]
Hence the same integrable bound $C(p)t^{-\alpha}$ holds throughout $(0,1]$.

In the critical and superdiffusive regimes, the preceding estimates at $t=1$ show that
\[
b_p(L)\PP(\sigma_L>L^2)\leq C(p).
\]
Monotonicity bounds $F_{L,p}^{\rm exit}$ uniformly on $[1,2]$, and Lemma~\ref{tailm2up1derwlem} gives
\[
F_{L,p}^{\rm exit}(t)\leq C(p)\rho^{\lfloor t\rfloor-1},\qquad t\geq2.
\]
Together with the bounds near zero, this gives the required integrable upper bound.
\end{proof}

\begin{corollary}
\label{corprofiletomean}
Fix $p\in[0,1)$ and $\bullet\in\{\mathrm{cov},\mathrm{exit}\}$. If
\[
\lim_{L\to\infty}F_{L,p}^{\bullet}(t)=F_p^{\bullet}(t)
\]
for Lebesgue-a.e. $t>0$, then $F_p^{\bullet}\in L^1(0,\infty)$ and
\begin{equation}
\label{profiletomean}
\lim_{L\to\infty}\frac{b_p(L)}{L^2}\EE\tau_L^{\bullet}
=\int_0^\infty F_p^{\bullet}(t)\,dt,
\end{equation}
where $\tau_L^{\rm cov}:=\tau_{\rm cov}$ and $\tau_L^{\rm exit}:=\sigma_L$.
\end{corollary}

\begin{proof}
By the tail-sum formula,
\[
\frac{b_p(L)}{L^2}\EE\tau_L^{\bullet}
=\int_0^\infty F_{L,p}^{\bullet}(t)\,dt.
\]
Proposition~\ref{proptailUI1d} allows us to apply the dominated convergence theorem.
\end{proof}

\subsection{The superdiffusive regime: a change of measure}
\label{secchange1d}

Fix $\alpha\in(1/2,1)$ and put
\[
\bar\alpha:=1-\alpha,\qquad \delta:=2\alpha-1=1-2\bar\alpha.
\]
Trajectories that remain in an interval of size $L$ up to time $L^2$ are rare under $\PP^{(\alpha)}$, but occur with positive limiting probability under the diffusive law $\PP^{(\bar\alpha)}$ (recall \cite[Theorem 1]{Andre2026}). We compare the two laws and normalize their likelihood ratio. Its limit will determine the constant in the original tail probability.

Recall that $X_n=\widetilde S_n-\widetilde S_{n-1}$. By \eqref{condtranStild}, the Radon--Nikodym derivative of the first $n$ increments under $\PP^{(\alpha)}$ with respect to $\PP^{(\bar\alpha)}$ is
\begin{equation}
\label{defRadonder}
\widetilde{\mathscr R}_n
:=\prod_{k=1}^{n-1}\frac{1+\alpha W_k}{1+\bar\alpha W_k},
\qquad
W_k:=\frac{X_{k+1}\widetilde S_k}{k},
\end{equation}
with $\widetilde{\mathscr R}_1:=1$. The complementary parameter $\bar\alpha=1-\alpha$ is distinguished by the cancellation in \eqref{exactlogZn} below.

\begin{proposition}
\label{propnormalizedlikelihood}
Under $\PP^{(\bar\alpha)}$, define
\begin{equation}
\label{defZnlikelihood}
\mathcal Z_n:=n^{\delta/2}\widetilde{\mathscr R}_n
\exp\left(-\frac{\delta\widetilde S_n^2}{2n}\right),\qquad n\geq1.
\end{equation}
Then $\mathcal Z_n$ converges almost surely and in $L^1(\PP^{(\bar\alpha)})$ to a random variable $\mathcal Z_\infty\in(0,\infty)$. Moreover, for some $q>1$,
\begin{equation}
\label{ZnLqbound}
\sup_{n\geq1}\EE^{(\bar\alpha)}\mathcal Z_n^q<\infty.
\end{equation}
In particular,
\begin{equation}
\label{defcalpha}
c_\alpha:=\EE^{(\bar\alpha)}\mathcal Z_\infty\in(0,\infty).
\end{equation}
\end{proposition}

\begin{proof}
Set $Q_k:=\widetilde S_k^2/k$ and
\[
r_\alpha(w):=\log\left(\frac{1+\alpha w}{1+\bar\alpha w}\right)-\delta w+\frac\delta2w^2.
\]
Since $\alpha+\bar\alpha=1$, Taylor's formula gives $|r_\alpha(w)|\leq C(\alpha)|w|^3$ for $w\in[-1,1]$. The identity
\[
\delta W_k-\frac\delta2W_k^2
=\frac\delta2\left(1+\frac1k\right)(Q_{k+1}-Q_k)-\frac\delta{2k}
\]
and Abel summation yield, for $n\geq2$,
\begin{equation}
\label{exactlogZn}
\begin{aligned}
\log\mathcal Z_n
={}&\frac\delta2\left[
\frac{Q_n-n}{n-1}
+\sum_{k=2}^{n-1}\frac{Q_k-1}{k(k-1)}
+\log n-\sum_{k=1}^{n-1}\frac1k
\right]
+\sum_{k=1}^{n-1}r_\alpha(W_k).
\end{aligned}
\end{equation}
Under $\PP^{(\bar\alpha)}$, the law of the iterated logarithm in \cite[Theorem 3.2]{MR3741953} gives
\[
|\widetilde S_n|=O(\sqrt{n\log\log n})\qquad\text{a.s.}
\]
Consequently, in \eqref{exactlogZn}, the two series
\[
\sum_{k=2}^{n-1}\frac{Q_k-1}{k(k-1)} \quad \text{and }\  \sum_{k=1}^{n-1}r_\alpha(W_k)
\]
 converge absolutely, while $(Q_n-n)/(n-1)\to-1$ and $\log n-\sum_{k<n}k^{-1}$ converges. Hence $\mathcal Z_n\to\mathcal Z_\infty\in(0,\infty)$ almost surely.

It remains to prove uniform integrability. Given $\FF_n$, write $x_n:=\widetilde S_n/n$ and $r_n:=n/(n+1)$. A direct calculation from \eqref{condtranStild} and \eqref{defZnlikelihood} gives
\begin{equation}
\label{Znqrecursion}
\frac{\EE^{(\bar\alpha)}(\mathcal Z_{n+1}^q\mid\FF_n)}{\mathcal Z_n^q}
=C_{n,q}\Psi_{q,r_n}(x_n),
\end{equation}
where
\[
C_{n,q}:=\left(1+\frac1n\right)^{q\delta/2}
\exp\left(-\frac{q\delta}{2(n+1)}\right)=\exp(O(n^{-2}))
\]
and
\[
\Psi_{q,r}(x):=\frac{e^{q\delta r x^2/2}}2
\sum_{\varepsilon=\pm1}
(1+\bar\alpha\varepsilon x)^{1-q}
(1+\alpha\varepsilon x)^q e^{-q\delta r\varepsilon x}, \quad (q,r,x) \in \RR \times \RR \times \left(-\frac{1}{\alpha},\frac{1}{\alpha}\right).
\]
For $q=r=1$,
\[
\Psi_{1,1}(x)=e^{\delta x^2/2}
\bigl(\cosh(\delta x)-\alpha x\sinh(\delta x)\bigr)\leq1.
\]
Indeed, if $\xi_x\in\{-1,1\}$ has law $\PP(\xi_x=\varepsilon)=(1+\alpha\varepsilon x)/2$, then the left-hand side is $e^{\delta x^2/2}\EE e^{-\delta x\xi_x}$. Since $\EE\xi_x=\alpha x$, Hoeffding's lemma gives
\[
\log\Psi_{1,1}(x)
\leq\left(\frac\delta2-\alpha\delta+\frac{\delta^2}{2}\right)x^2=0.
\]
The inequality is strict for $x\neq 0$ (see e.g. \cite[the proof of Lemma 2.2]{boucheron2013concentration}), and
\begin{equation}
\label{Psiquartic}
\log\Psi_{1,1}(x)
=-\frac{\delta^2(\delta^2+2\delta+3)}{24}x^4+O(x^6)
\qquad(x\to0).
\end{equation}
Observe that $\log \Psi_{q,r}(x)$ is smooth in $(q,r,x)$ on a neighborhood of $[0,q_0] \times [1/2,1] \times [-1,1]$ for any $q_0>1$. For each fixed $(q,r)$, it is even in $x$ and satisfies $\Psi_{q,r}(0)=1$. Moreover, since $\alpha+\bar\alpha=1$ and $\alpha-\bar\alpha=\delta$,
\[
\partial_x^2\log\Psi_{q,1}(0)
=q(\delta+\bar\alpha^2-\alpha^2)=0
\]
for every $q$. Thus the expansion of $\log\Psi_{q,1}$ starts at order $x^4$, and by \eqref{Psiquartic} its quartic coefficient is negative at $q=1$. Now consider $f_{q}(x):=(\log\Psi_{q,1}(x))/x^4$ for $(q,x) \in [0,q_0] \times [-1,1]$ where $f_q(0):=\lim_{x\to 0} f_q(x)$. Note that $f_1(x)$ is negative and bounded away from $0$ for $x \in [-1,1]$ by continuity in $x$ and compactness. Since $f_q(x)$ is continuous in $(q,x)$, we may choose $q>1$ sufficiently close to $1$ and $c>0$ such that
\begin{equation}
\label{logPsiq1bound}
\log\Psi_{q,1}(x)\leq-cx^4, \qquad |x|\leq 1.
\end{equation}

It remains to control the dependence on $r$. Write $w_\varepsilon:=(1+\bar\alpha\varepsilon x)^{1-q}(1+\alpha\varepsilon x)^q e^{-q\delta r\varepsilon x}$, so that $\Psi_{q,r}(x)=\frac{e^{q\delta rx^2/2}}2(w_++w_-)$. Differentiating in $r$ gives
\[
\partial_r\log\Psi_{q,r}(x)
=\frac{q\delta x^2}{2}-q\delta x\,\frac{w_+-w_-}{w_++w_-}.
\]
Now $w_+-w_-$ is odd in $x$, while $w_++w_-$ is even and bounded away from $0$ on $[-1,1]\times[1/2,1]$, hence, as $x \to 0$, the ratio $(w_+-w_-)/(w_++w_-)$ is $O(x)$, and $\partial_r\log\Psi_{q,r}(x)=O(x^2)$ uniformly in $r\in[1/2,1]$. As $r_n-1=-\frac1{n+1}$, the mean value theorem yields
\begin{equation}
\label{logPsirdiff}
|\log\Psi_{q,r_n}(x)-\log\Psi_{q,1}(x)|
\leq \frac{C}{n}x^2.
\end{equation}
Combining \eqref{logPsiq1bound} and \eqref{logPsirdiff}, and setting $y=x^2$,
\[
\log\Psi_{q,r_n}(x)
\leq-cx^4+\frac Cn x^2
=-cy^2+\frac Cn y
\leq \max_{y\geq0}\Bigl(-cy^2+\frac Cn y\Bigr)
=\frac{C^2}{4c\,n^2}
=:C_1n^{-2}.
\]
By \eqref{Znqrecursion} and $C_{n,q}=\exp(O(n^{-2}))$, this gives
\[
\EE^{(\bar\alpha)}(\mathcal Z_{n+1}^q\mid\FF_n)
\leq C_{n,q}e^{C_1n^{-2}}\mathcal Z_n^q
\leq e^{C_2/n^2}\mathcal Z_n^q.
\]
Iteration proves \eqref{ZnLqbound}. The almost-sure convergence is therefore also in $L^1$, completing the proof.
\end{proof}

Recall the noise-reinforced Brownian motion in Proposition~\ref{prop1dlimit}(i). Under the reference law $\PP^{(\bar\alpha)}$, the memory parameter is $p_{\rm ref}=(1+\bar\alpha)/2<3/4$, so $2p_{\rm ref}-1=1-\alpha$ and $3-4p_{\rm ref}=\delta$. We denote the corresponding limiting process by $G^{(\alpha)}$ and realize it as
\begin{equation}
\label{defGalphaintro}
G_t^{(\alpha)}:=\frac{t^{1-\alpha}}{\sqrt\delta}B_{t^\delta},\qquad t\geq0.
\end{equation}
Recall the positive constant $c_\alpha$ defined in \eqref{defcalpha}. Define
\begin{align}
\label{defHexit}
\mathcal H_\alpha^{\rm exit}(t)
&:=c_\alpha t^{-\delta/2}
\EE\left[
\exp\left(\frac{\delta(G_t^{(\alpha)})^2}{2t}\right)
\1_{\{\sup_{0\leq s\leq t}|G_s^{(\alpha)}|<1\}}
\right],\\
\label{defHcov}
\mathcal H_\alpha^{\rm cov}(t)
&:=c_\alpha t^{-\delta/2}
\EE\left[
\exp\left(\frac{\delta(G_t^{(\alpha)})^2}{2t}\right)
\1_{\{\Osc_t(G^{(\alpha)})<1\}}
\right].
\end{align}

The following Proposition \ref{thmsuperprofiles} implies Proposition \ref{propcover1d} (ii).

\begin{proposition}
\label{thmsuperprofiles}
For every $t>0$,
\begin{align}
\label{superexitprofile}
\lim_{m\to\infty}m^\delta\PP^{(\alpha)}(\sigma_m>\lfloor m^2t\rfloor)
&=\mathcal H_\alpha^{\rm exit}(t),\\
\label{supercoverprofile}
\lim_{L\to\infty}L^\delta\PP^{(\alpha)}(\tau_{\rm cov}>\lfloor L^2t\rfloor)
&=\mathcal H_\alpha^{\rm cov}(t).
\end{align}
Consequently,
\begin{align}
\label{superexitmean}
\lim_{m\to\infty}\frac{\EE^{(\alpha)}\sigma_m}{m^{2-\delta}}
&= C_{\rm exit}(p):=\int_0^\infty\mathcal H_\alpha^{\rm exit}(t)\,dt,\\
\label{supercovermean}
\lim_{L\to\infty}\frac{\EE^{(\alpha)}\tau_{\rm cov}}{L^{2-\delta}}
&= C_+(p):=\int_0^\infty\mathcal H_\alpha^{\rm cov}(t)\,dt.
\end{align}
The limits in \eqref{superexitmean} and \eqref{supercovermean} are positive and finite.
\end{proposition}

\begin{proof}
Applying the stable functional central limit theorem in \cite[Theorem 1.3]{Tokumitsu2026stable} with memory parameter $p_{\rm ref}=(1+\bar\alpha)/2<3/4$ and auxiliary variable $\mathcal Z_\infty$ in Proposition \ref{propnormalizedlikelihood}, one has, under $\PP^{(\bar\alpha)}$, in the product space $D([0,\infty)) \times \RR$,
\begin{equation}
\label{stableFCLTreference}
\left(\left(\frac{\widetilde S_{\lfloor ns\rfloor}}{\sqrt n}\right)_{s\geq0},\mathcal Z_\infty\right)
\Longrightarrow\left((G_s^{(\alpha)})_{s\geq0},\mathcal Z_\infty\right),
\end{equation}
where, on the right-hand side, $G^{(\alpha)}$ is independent of $\mathcal Z_\infty$. We use this convergence with $n=m^2$. Fix $t>0$ and put $N_m:=\lfloor m^2t\rfloor$. By change of measure and \eqref{defZnlikelihood},
\begin{equation}
    \label{eqmdelPalpha}
m^\delta\PP^{(\alpha)}(\sigma_m>N_m)
=\left(\frac{m^2}{N_m}\right)^{\delta/2}
\EE^{(\bar\alpha)}\left[
\mathcal Z_{N_m}
\exp\left(\frac{\delta\widetilde S_{N_m}^2}{2N_m}\right);
\sigma_m>N_m \right].
\end{equation}
On $\{\sigma_m>N_m\}$ we have $|\widetilde S_{N_m}|<m$, hence, 
\[
1\leq\exp\Bigl(\frac{\delta\widetilde S_{N_m}^2}{2N_m}\Bigr)
\leq\exp\Bigl(\frac{\delta m^2}{2N_m}\Bigr)\leq C(t,\alpha).
\]
By Proposition~\ref{propnormalizedlikelihood}, $\mathcal Z_n\to\mathcal Z_\infty$ in $L^1(\PP^{(\bar\alpha)})$. Combined with the uniform bound above, replacing $\mathcal Z_{N_m}$ by $\mathcal Z_\infty$ in (\ref{eqmdelPalpha}) changes the right-hand side by at most
\[
C(t,\alpha)\,\EE^{(\bar\alpha)}|\mathcal Z_{N_m}-\mathcal Z_\infty|=o(1).
\]
It remains to compute the resulting limit. Write $\phi_m:=(m^{-1}\widetilde S_{\lfloor m^2s\rfloor})_{s\geq0}$. For $\phi\in D([0,\infty))$ with $\phi(0)=0$, define
\[
f_t(\phi):=\exp\left(\frac{\delta\phi(t)^2}{2t}\right)
\1_{\{\sup_{0\leq s\leq t}|\phi(s)|<1\}}.
\]
This functional is bounded. Convergence in the Skorokhod topology to a continuous path implies uniform convergence on compact time intervals. Hence $f_t$ is continuous at every continuous path outside the boundary event $\{\sup_{s\leq t}|\phi(s)|=1\}$. On $C_0([0,t]):=\{\phi\in C([0,t]):\phi(0)=0\}$, both $\|\phi\|_\infty$ and $\Osc_t(\phi)$ are norms, with
\[
\|\phi\|_\infty\leq\Osc_t(\phi)\leq2\|\phi\|_\infty.
\]
The norm of a nontrivial centered Gaussian random element has no atoms at positive values, see e.g. \cite{Rhee1984}. Thus, the boundary events for both the supremum and the oscillation have probability zero under the law of $G^{(\alpha)}$.

Apply \eqref{stableFCLTreference} with $\mathcal Z_\infty$ truncated at level $K$. Since $N_m/m^2\to t$, we obtain
\[
\lim_{m\to\infty}\EE^{(\bar\alpha)}
\bigl[(\mathcal Z_\infty\wedge K)f_t(\phi_m)\bigr]=\EE^{(\bar\alpha)}[\mathcal Z_\infty\wedge K]\, \EE[f_t(G^{(\alpha)})].
\]
Since $|\phi_m(t)|<1$ and $|\widetilde S_{N_m}|<m$ on $\{\sigma_m>N_m\}$ and $\mathcal Z_\infty$ is integrable, dominated convergence implies that as $m \to \infty$,
\[
\EE^{(\bar\alpha)}\left[ (\mathcal Z_\infty\wedge K) \left| \exp\left(\frac{\delta\widetilde S_{N_m}^2}{2N_m}\right) -\exp\left(\frac{\delta\widetilde S_{N_m}^2}{2m^2 t}\right)\right|; \sigma_m>N_m \right] = o(1),
\]
 uniformly in $K$.
The truncation error, before and after passage to the limit, is bounded by
\[
C(t,\alpha)\EE^{(\bar\alpha)}[(\mathcal Z_\infty-K)_+]\longrightarrow0
\qquad(K\to\infty),
\]
since $\mathcal Z_\infty$ is integrable. Letting $K\to\infty$ and recalling that $(m^2/N_m)^{\delta/2}\to t^{-\delta/2}$ proves \eqref{superexitprofile}.

For \eqref{supercoverprofile}, we use the convergence (\ref{stableFCLTreference}) with $n=L^2$. Put $N_L:=\lfloor L^2t\rfloor$. On $\{\tau_{\rm cov}>N_L\}$, at time $N_L$, the range of $\tilde{S}/L$ is smaller than $1-1/L$ by the definition \eqref{coverrangeidentity}, and in particular, $|\tilde{S}_{N_L}| \leq L$. The same argument as above applies, and the boundary event $\{\Osc_t(G^{(\alpha)})=1\}$ has probability zero. Finally, \eqref{superexitmean} and \eqref{supercovermean} follow from Corollary~\ref{corprofiletomean}. By \eqref{defGalphaintro} and the positive probability that Brownian motion stays in a bounded interval, for every $t>0$,
\[
\PP\left(\sup_{0\leq s\leq t}|G_s^{(\alpha)}|<\frac12\right)>0.
\]
This event is contained in both events in \eqref{defHexit} and \eqref{defHcov}. Thus, both constants are positive.
\end{proof}

\subsection{The critical regime: embedding into a Brownian motion} 
\label{seccritBM}
We assume that $p=3/4$, i.e., $\alpha=1/2$, in this subsection. It is direct to check that $M=(M_n)$ defined below is a martingale. For $n\geq1$,
\begin{equation}
  \label{martMnp34}
 M_n:=a_n \widetilde S_n, \quad \text{where}\ a_n: =\frac{\Gamma(n)}{\Gamma(n+1/2)} = \frac{1}{\sqrt{n}} + O\left(\frac{1}{n}\right),
\end{equation}
with the convention that $M_0=0$. Fang \cite{Fang2024} showed that $M$ can be embedded into a standard Brownian motion $B=(B_t)_{t\geq 0}$ as follows: Set $T_0 := 0$ and define recursively
\[
T_{n+1}=\inf\left\{t> T_n: B_t-B_{T_n} = -\frac{B_{T_n}}{2n+1}\pm a_{n+1}\right\}, \quad n \in \NN.
\]
Then, 
\[
(M_n)_{n\in \NN} \stackrel{(d)}{=} (B_{T_n})_{n\in \NN}.
\]
Without loss of generality, we shall henceforth work with a realization of $\widetilde S_n$ such that the above relation holds almost surely. Since $|\tilde S_n| \leq n$, one has 
\begin{equation}
    \label{estMBan}
   \frac{|B_{T_n}|}{2n+1} \leq \frac{a_{n+1}}{2}, \quad |M_{n+1}-M_{n}|\leq 2a_{n+1}
\end{equation}

Define $A_0:=0$ and $A_n:= \sum_{k=1}^n a_k^2$, $n \geq 1$. By \eqref{martMnp34}, as $n\to \infty$, 
\[
    A_n = \log n + C+   O\left(\frac{1}{\sqrt{n}}\right).
\]
 Thus, there exists a positive constant $C$ such that for any $1< \ell \leq n$, one has
\begin{equation}
    \label{Ankest}
|A_n -A_{\ell} -(\log n - \log \ell)|  \leq  \frac{C}{\sqrt{\ell}}.
\end{equation}

The following lemma gives a conditional estimate for the increments of $T_n-A_n$. We write $(\FF_t^B)_{t\geq 0}$ for the natural filtration of $B$.
\begin{lemma}
  \label{lemconditionalprop42}
  For any positive integer $r$, there exists a positive constant $C(r)$ such that for any integers $1< \ell \leq n$,
  \[
  \mathbb{E}\left(\sup _{\ell \leq k \leq n}\left|\left(T_k-T_{\ell}\right)-\left(A_k-A_{\ell}\right)\right|^{2 r}\mid \FF_{T_\ell}^B\right)     \leq C(r)\left(\ell^{-r}+\left(\frac{1+\left|B_{T_{\ell}}\right|^2}{\ell}\right)^{2 r}\right)
  \]
\end{lemma}
\begin{proof}
We adapt the argument in the proof of \cite[Proposition 4.2]{Fang2024} to increments after $T_\ell$, using the strong Markov property. The Burkholder--Davis--Gundy inequality and \eqref{estMBan} and \eqref{Ankest} give, for any $k\geq\ell$,
  \[
  \EE ((M_k-M_\ell)^{2r} \mid \FF_{T_\ell}^B)  \leq C(r) (A_k-A_\ell)^r \leq C(r) (1+\log(k/\ell))^r.
  \]
Since $B_{T_k}=B_{T_k}-B_{T_\ell}+B_{T_\ell}$, one has 
 \begin{equation}
  \label{estMk2rcond}
 \EE (B_{T_k}^{2r} \mid \FF_{T_\ell}^B) =\EE (M_k^{2r} \mid \FF_{T_\ell}^B) \leq C(r) (|B_{T_\ell}|^2 + 1+\log(k/\ell))^r.
 \end{equation}
 Define 
 \[
V_n:=\sum_{j=0}^{n-1} \frac{1}{(2j+1)^2} B^2_{T_j}, \quad n\in \NN.
 \]
Then the (conditional) Minkowski inequality and \eqref{estMk2rcond} give
\[
\begin{aligned}
\left(\EE\left((V_n-V_\ell)^r\mid\FF^B_{T_\ell}\right)\right)^{1/r}
&\leq\sum_{j=\ell}^{n-1}
\frac{\left(\EE\left(|B_{T_j}|^{2r}\mid\FF^B_{T_\ell}\right)\right)^{1/r}}{(2j+1)^2}\\
&\leq C(r)\sum_{j=\ell}^{n-1}
\frac{|B_{T_\ell}|^2+1+\log(j/\ell)}{j^2}\\
&\leq C_1(r)\frac{1+|B_{T_\ell}|^2}{\ell},
\end{aligned}
\]
where in the last inequality we used that
 \[
\sum_{j=\ell}^{\infty}\frac{\log (j / \ell)}{j^2}=\sum_{j=\ell}^{\infty}\frac{1}{j^2}\int_1^{j/\ell}\frac{du}{u} =\int_1^{\infty}\sum_{j\ge \ell u}\frac{1}{j^2}\frac{du}{u} \leq \int_1^{\infty} \frac{2}{\ell u^2} du =\frac{2}{\ell}.
 \]
Conditionally on $\FF^B_{T_{j-1}}$, the increment $T_j-T_{j-1}$ is the exit time of Brownian motion from an interval of length $2a_j$. Its conditional mean is
\[
\EE(T_j-T_{j-1}\mid\FF^B_{T_{j-1}})
=a_j^2-\frac{B_{T_{j-1}}^2}{(2j-1)^2}.
\]
Consequently,
\[
\mathscr N_n:=T_n-A_n+V_n,\qquad n\in\NN,
\]
is a martingale. By \eqref{estMBan}, the conditional mean displayed above lies between $0$ and $a_j^2$, so
\[
\begin{aligned}
|\mathscr N_j-\mathscr N_{j-1}|
&=\left|T_j-T_{j-1}-a_j^2+\frac{B^2_{T_{j-1}}}{(2j-1)^2}\right|\\
&\leq T_j-T_{j-1}+a_j^2.
\end{aligned}
\]
Brownian scaling and the uniform moment bounds for exit times from a bounded interval give
\[
\EE\left(|\mathscr N_j-\mathscr N_{j-1}|^{2r}
\mid\FF^B_{T_{j-1}}\right)\leq C(r)a_j^{4r},
\]
see \cite[Lemma 4.4 and (4.2)]{MR4469228} for more details. The conditional Burkholder--Davis--Gundy and Minkowski inequalities therefore yield
\[
\begin{aligned}
\EE\left(\sup_{\ell\leq k\leq n}|\mathscr N_k-\mathscr N_\ell|^{2r}
\mid\FF^B_{T_\ell}\right)&\leq C(r)\EE\left(
\left(\sum_{j=\ell+1}^n|\mathscr N_j-\mathscr N_{j-1}|^2\right)^r
\mid\FF^B_{T_\ell}\right)\\
&\leq C_1(r)\left(\sum_{j=\ell+1}^n a_j^4\right)^r
\leq C_2(r)\ell^{-r}.
\end{aligned}
\]
This is the conditional version of the estimate used in \cite[Lemma 4.6]{MR4469228}.
Applying the estimate for $V_n-V_\ell$ with $2r$ in place of $r$, the desired result follows from
\[
(T_k-T_\ell)-(A_k-A_\ell)=\mathscr N_k-\mathscr N_\ell-(V_k-V_\ell).  \qedhere
\]
\end{proof}

We shall also use the following result on Brownian motion. 

\begin{lemma}
\label{BMlemtube}
  Fix $b>0$. Let $W=x+B$ be a Brownian motion starting at $x \in \RR$. Then there exist positive constants $t(b),c(b),C(b)$ such that for all $t>t(b)$ and all $x\in [-\sqrt{t},\sqrt{t}]$, we have 
  \[
\frac{c(b)}{\sqrt{t}} \leq \PP(|W_s|<b e^{(t-s)/2} \text{ for all } s\in [0,t])  \leq \PP(|W_t|<b ) \leq \frac{C(b)}{\sqrt{t}}. 
  \]
\end{lemma}
\begin{proof}
The transition density of Brownian motion is
\begin{equation}
    \label{tranBMpt}
    p_t(x,y) = \frac{1}{\sqrt{2\pi t}} e^{-(y-x)^2/(2t)},
\end{equation}
and is bounded above by $1/\sqrt{t}$. This gives the upper bound. It remains to prove the lower bound. We can write 
\begin{equation}
    \label{Wtsdecom}
    W_{t-s}=x+B_{t-s}=W_t+\frac{s}{t}(x-W_t) + \beta^{(t)}_s, \quad s\in [0,t],
\end{equation}
where
\[
\beta^{(t)}_s := B_{t-s}-\frac{(t-s)}{t}B_t, \quad s\in [0,t],
\]
  is a Brownian bridge from $0$ to $0$ over the time interval $[0,t]$. Observe that $\sup_{s \geq 0} s e^{-s/2} \leq b \sqrt{K}/8$ for some constant $K=K(b)$. If $t\geq K$ and $W_t=x+B_t=y \in [-b/8,b/8]$, then for all $s\in [0,t]$ and all $x\in [-\sqrt{t},\sqrt{t}]$, 
\[
\left|W_t+\frac{s}{t}(x-W_t)\right| \leq \frac{t-s}{t}|y| + \frac{s}{t}|x| \leq \frac{b}{8} + \frac{s}{\sqrt{t}} \leq \frac{b}{4}e^{s/2}.
\]
In particular, by \eqref{Wtsdecom}, $|W_{t-s}| < b e^{s/2}$ for all $s\in [0,t]$ on the event $E_1 \cap E_2$ where
\[
E_1:= \left\{|\beta^{(t)}_s|\leq\frac{b}{2} e^{s/2} \text{ for all } s\in [0,t]\right\} \quad \text{and}\quad E_2:= \left\{|W_t|\leq \frac{b}{8}\right\}.
\]

Observe that $(\beta^{(t)}_s)_{0\leq s \leq t}$ has the same distribution as $(\tilde{\beta}^{(t)}_s)_{0\leq s \leq t}$ where $\tilde{\beta}^{(t)}_s:=B_s-sB_t/t$. Let $0=s_0<s_1<s_2<\dots<s_m=t$ be a partition of $[0,t]$, then applying Anderson inequality (see \cite[Theorem 2 or Corollary 3]{Anderson1955}) to the vector
  \[
(B_{s_0}, B_{s_1}, \dots, B_{s_m}) = (\tilde{\beta}^{(t)}_{s_0}, \tilde{\beta}^{(t)}_{s_1}, \dots, \tilde{\beta}^{(t)}_{s_m}) +B_t \left(\frac{s_0}{t},\frac{s_1}{t}, \dots, \frac{s_m}{t}\right), 
  \]
  and using that $(\tilde{\beta}^{(t)}_s)_{0\leq s \leq t}$ is independent of $B_t$, we have, for any $t>0$,
\begin{equation}
\label{eqandersonbridge}
\begin{aligned}
\PP(|\beta^{(t)}_{s_j}|\leq \frac{b}{2} e^{s_j/2}, 0\leq j \leq m) &=\PP(|\tilde{\beta}^{(t)}_{s_j}|\leq \frac{b}{2} e^{s_j/2}, 0\leq j \leq m) \\
&\geq \PP(|B_{s_j}|\leq \frac{b}{2} e^{s_j/2}, 0\leq j \leq m) \geq \PP(|B_s| \leq \frac{b}{2} e^{s/2}, \forall s\geq 0). 
\end{aligned}
\end{equation}
The last probability in \eqref{eqandersonbridge} is positive. Indeed, since $\sup_{s\geq0}e^{-s/2}|B_s|<\infty$ almost surely, we may choose $T$ sufficiently large that
\[
\PP\left(|B_{T+s}-B_T|<\frac b4e^{(T+s)/2}\text{ for all }s\geq0\right)\geq\frac12.
\]
This event depends only on the increments after time $T$ and is independent of
\[
\left\{\sup_{0\leq s\leq T}|B_s|<\frac b4\right\},
\]
which has positive probability. On their intersection, $|B_s|<(b/2)e^{s/2}$ for every $s\geq0$. Taking a nested sequence of partitions of $[0,t]$ whose mesh tends to zero therefore gives $\PP(E_1)\geq c(b)>0$. For any $b>0$, by (\ref{tranBMpt}), there exists a constant $t(b)>K$ such that for all $t>t(b)$ and any $x\in [-\sqrt{t},\sqrt{t}]$ and any $y \in [-b/8,b/8]$, one has $p_t(x,y) \geq c(b) t^{-1/2}$. Integrating over $y$ yields that $\PP(E_2) \geq c(b) t^{-1/2}$. It remains to note that $E_1$ and $E_2$ are independent.
\end{proof}

We now prove the lower bound. The idea is as follows: Choose a large starting time $\ell$ at which $|M_\ell|$ is controlled. In view of Lemma~\ref{lemconditionalprop42}, we may approximate $T_k-T_\ell$ by $\log(k/\ell)$. Writing $s\approx\log(k/\ell)$, the constraint $|\widetilde S_k|<m$ for all $k\in [\ell,m^2]$ becomes a boundary of size
\[
ma_k\asymp\frac m{\sqrt\ell}e^{-s/2} = e^{(2\log m - \log \ell -s)/2}
\]
for the embedded Brownian motion. This is the tube used in Lemma~\ref{BMlemtube}.

\begin{proof}[Proof of Proposition \ref{propcover1d}(i)] 
The upper bound follows from the tail-sum formula and Proposition \ref{proptailUI1d}. To prove the lower bound, it suffices to show that for all $m>1$, one has 
\begin{equation}
    \label{tailtaumlbcrit}
    \PP(\sigma_m \geq m^2) \geq \frac{C}{\sqrt{\log m}}.
\end{equation}
We may assume that $m$ is sufficiently large such that $m> \ell:=\lceil \log^8 m \rceil$. It is known that (see \cite[Lemma 2.1]{Fang2024}) $\EE M_n^2 \leq C \log n$ for all $n>1$. By Markov inequality, for all large $m$, 
  \[
  \PP(E) \geq \frac{1}{2} \quad \text{where } E:=\{|M_{\ell}| \leq \sqrt{t}\} \text{ and } t:=2\log m -\log \ell+2.
  \]
  Since $\inf_{n\geq 1} a_n \sqrt{n} \geq \hat{c}$ for some positive constant $\hat{c}$, we let $b:= \hat{c}/(2 e^2)$ and
 define the events
  \[
  E_1:=\{\sup _{\ell \leq k \leq m^2 }\left|\left(T_k-T_{\ell}\right)-\left(A_k-A_{\ell}\right)\right| \leq 1\} \text{ and } E_2:=\{ |B_{T_{\ell}+s}| \leq b e^{(t-s)/2} \text{ for all } s\in [0,t]\}.
  \]
  On the event $E$, Lemma \ref{lemconditionalprop42} (with $r=1$) implies that 
  \[
  \PP(E_1^c \mid \FF^{B}_{T_{\ell}}) \1_{E} \leq  C_1\left(\frac{1}{\ell}+\left(\frac{1+2 \log m}{\ell}\right)^{2}\right) \1_{E} \leq \frac{C_2}{\log^8 m}\1_{E},
  \]
  and by Lemma \ref{BMlemtube},  $\PP(E_2 \mid \FF^{B}_{T_{\ell}}) \1_{E} \geq c(b) (\log m)^{-1/2}\1_{E}$ (note that $t =\Theta(\log m)$). Since
\[
\PP(E_1\cap E_2\mid\FF^B_{T_\ell})
\geq\PP(E_2\mid\FF^B_{T_\ell})-\PP(E_1^c\mid\FF^B_{T_\ell}),
\]
it follows, for all sufficiently large $m$, that
\[
\PP(E\cap E_1\cap E_2) = \EE(\PP(E_1\cap E_2\mid\FF^B_{T_\ell}) \mathbf1_E) \geq\frac{c(b)}{\sqrt{\log m}}.
\]
On the event $E_1 \cap E_2$, for all $\ell \leq k \leq m^2$ where $m$ is sufficiently large, by \eqref{Ankest}, one has
\[
\log k -\log \ell -2 \leq A_k-A_{\ell} -1 \leq  T_k - T_{\ell} \leq A_k-A_{\ell} +1 \leq t,
\]
 and thus, 
  \begin{equation}
      \label{estTilSkbem2}
      |\widetilde S_k|= \frac{|M_k|}{a_k}=\frac{|B_{T_k}|}{a_k} \leq  \frac{b  e^{(t-T_k+T_{\ell})/2}}{ a_k}  \leq   \frac{b e^2 m}{a_k \sqrt{k}} \leq \frac{m}{2}.
  \end{equation}
Since $|\widetilde S_k| < \ell < m$ for $k<\ell$, this shows that $E_1 \cap E_2 \subset \{\sigma_m \geq m^2\}$.
\end{proof}

\begin{question}[Mean cover time at criticality]
For $d=1$ and $p=3/4$, does there exist a constant $C_{\rm crit}\in(0,\infty)$ such that
\[
\lim_{L\to\infty}\frac{\sqrt{\log L}}{L^2}\EE\tau_{\rm cov}
=C_{\rm crit}?
\]
\end{question}

By Corollary~\ref{corprofiletomean}, it would suffice to prove that, for some function $\mathcal H_{\rm crit}$ and Lebesgue-a.e. $t>0$,
\[
\lim_{L\to\infty}\sqrt{\log L}\,
\PP(\tau_{\rm cov}>\lfloor L^2t\rfloor)
=\mathcal H_{\rm crit}(t).
\]
The constant would then be $C_{\rm crit}=\int_0^\infty\mathcal H_{\rm crit}(t)\,dt$.

\subsection{Fluctuation limits and completion of the proof}
\label{secfluct1d}

For $p>3/4$, \cite[Theorem 3]{baur2016elephant} implies that, as $n\to \infty$, almost surely and locally uniformly in $t\geq0$,
\begin{equation}
\label{as1dERWY}
\left(\frac{\widetilde S_{\lfloor nt\rfloor}}{n^{2p-1}},t\geq0\right)
\longrightarrow (t^{2p-1}Y,t\geq0),
\end{equation}
where $Y$ is a non-degenerate random variable. The law of $Y$ has been studied further by Gu\'erin et al.\ \cite{MR5079601, MR4926011}. In particular, it has a positive bounded smooth density \cite[Theorem 1.3]{MR5079601}.

\begin{proof}[Proof of Proposition \ref{prop1dlimit}]
We first consider $p\leq3/4$. Baur and Bertoin \cite[Theorem 1]{baur2016elephant} and Majumdar and Maulik \cite[Corollary 4.9]{Majumdar2025} proved the following convergence in $D([0,\infty))$:
\begin{equation}
\label{SkorkhodERW}
\left(\frac{\widetilde S_{\lfloor nt\rfloor}}{\sqrt n},t\geq0\right)\Longrightarrow(\widehat B_t,t\geq0),\quad p<\frac34,
\end{equation}
and
\begin{equation}
\label{SkorkhodERWcritical}
\left(\frac{\widetilde S_{\lfloor nt\rfloor}}{\sqrt{n\log n}},t\geq0\right)
\Longrightarrow(\sqrt t B_1,t\geq0),\quad p=\frac34.
\end{equation}
For $p<3/4$, take $n(L)=L^2$. For $p=3/4$, take $n(L)=\lfloor L^2/\log L\rfloor$. Then \eqref{SkorkhodERW}--\eqref{SkorkhodERWcritical} and Slutsky's theorem give
\[
\left(\frac{\widetilde S_{\lfloor n(L)t\rfloor}}{L},t\geq0\right)
\Longrightarrow
\begin{cases}
(\widehat B_t,t\geq0),&p<3/4,\\
(\sqrt{2t}B_1,t\geq0),&p=3/4,
\end{cases}
\]
since $\sqrt{n(L)\log n(L)}/L\to\sqrt2$ in the critical case.

The map $\Phi$ in \eqref{defPhi} is continuous at $(\lambda,f)$ whenever $f$ is continuous, $\Phi(\lambda,f)<\infty$, and the range of $f$ crosses the level $\lambda$ at that time. For $\widehat B$, finiteness  of $\Phi(1,\widehat B)$ follows from $\{\Phi(1,\widehat B)>t\}\subseteq\{|\widehat B_t|<1\}$. On every interval bounded away from zero, $\widehat B$ is a diffusion with a nonzero Brownian coefficient, so it immediately crosses either boundary upon hitting it. For the critical limit $\sqrt{2t}B_1$, the range is strictly increasing almost surely. Thus the required conditions hold, and the continuous mapping theorem (see e.g. \cite[Theorem 2.7]{MR1700749}) and \eqref{coverrangeidentity} yield
\[
\frac{\tau_{\rm cov}}{n(L)}
=\Phi\left(1-\frac1L,\left(\frac{\widetilde S_{\lfloor n(L)t\rfloor}}{L}\right)_{t\geq0}\right)
\Longrightarrow
\begin{cases}
\Phi(1,\widehat B),&p<3/4,\\
(2B_1^2)^{-1},&p=3/4.
\end{cases}
\]

Finally, suppose $p>3/4$ and let $n(L)=\lfloor L^{1/(2p-1)}\rfloor$. By \eqref{as1dERWY} and the almost-sure continuity of $\Phi$ at $(1,(t^{2p-1}Y)_{t\geq0})$,
\[
\frac{\tau_{\rm cov}}{n(L)}
=\Phi\left(1-\frac1L,\left(\frac{\widetilde S_{\lfloor n(L)t\rfloor}}{L}\right)_{t\geq0}\right)
\longrightarrow\frac{1}{|Y|^{1/(2p-1)}}
\]
almost surely. This proves the proposition.
\end{proof}

\begin{proof}[Completion of the proof of Theorem~\ref{thmcover}]
Suppose first that $p<3/4$. Proposition~\ref{prop1dlimit}(i) gives convergence of the tail probabilities at every continuity point of the limiting distribution. Corollary~\ref{corprofiletomean} therefore gives
\begin{equation}
\label{subcritmeanconstant}
\frac{\EE\tau_{\rm cov}}{L^2}
\longrightarrow \EE\Phi(1,\widehat B)=C_-(p).
\end{equation}
For $p=3/4$, the assertion follows from \eqref{equcovexit1D} and Proposition~\ref{propcover1d}(i). For $p>3/4$, it is precisely \eqref{supercovermean}, since $2-\delta=5-4p$.
\end{proof}

In dimension one, Akimoto, Takei and Taniguchi \cite{AKIMOTO2025110436} studied the range of the ERW with stops, including the usual ERW. They obtained almost sure limsup bounds in the diffusive and critical regimes and an almost sure scaling limit in the superdiffusive regime. The functional limit theorems above also give the following distributional limits.

\begin{corollary}
\label{cor1drangelimit}
Let $d=1$, with the notation of Proposition~\ref{prop1dlimit}.
\begin{enumerate}[(i)]
\item If $p<3/4$, then
\[
\left(\frac{|\mathcal R_{\lfloor nt\rfloor}|}{\sqrt n}\right)_{t\geq0}
\Longrightarrow
\bigl(\Osc_t(\widehat B)\bigr)_{t\geq0}.
\]
\item If $p=3/4$, then
\[
\left(\frac{|\mathcal R_{\lfloor nt\rfloor}|}{\sqrt{n\log n}}\right)_{t\geq0}
\Longrightarrow
\bigl(\sqrt t\,|B_1|\bigr)_{t\geq0}.
\]
\end{enumerate}
Both convergences hold as $n\to\infty$ in $D([0,\infty))$ with the Skorokhod topology.
\end{corollary}
\begin{proof}
Since the lifted walk is nearest-neighbor on $\ZZ$, it visits every integer between its minimum and maximum. Hence, with $f_n(s)=\widetilde S_{\lfloor ns\rfloor}$,
\[
|\mathcal R_{\lfloor nt\rfloor}|=1+\Osc_t(f_n),\qquad t\geq0.
\]
For every $T>0$ and $f,g\in D([0,T])$,
\[
\sup_{t\leq T}|\Osc_t(f)-\Osc_t(g)|
\leq2\sup_{s\leq T}|f(s)-g(s)|.
\]
Convergence in the Skorokhod topology to a continuous path implies uniform convergence on compact time intervals. The oscillation map is therefore continuous at every continuous path. Applying the continuous mapping theorem to \eqref{SkorkhodERW} and \eqref{SkorkhodERWcritical}, and noting that the added $1$ vanishes under either normalization, proves the two assertions. In the critical case, the limiting path is $s\mapsto\sqrt s B_1$, whose oscillation on $[0,t]$ is $\sqrt t\,|B_1|$.
\end{proof}
\section{Cover times in higher dimensions}
\label{secHighD}

Throughout this section, $d\geq2$ and $p<1$ are fixed. We first give the entropy estimates used to compare the ERW with simple random walk. Comparison through Radon--Nikodym derivatives was also used by Curien and Laulin~\cite[Proposition 2]{MR4824917} to study recurrence of the planar ERW. We then prove an upper-tail bound, using the conditional forest representation when $\alpha>0$ and entropy when $\alpha\leq0$. These estimates also provide the uniform integrability needed for Theorem~\ref{thmhighd}(i). In dimensions $d\geq3$, we then refine the hitting estimate to obtain an improved upper bound, which also applies at and above criticality.

\subsection{Entropy estimates on a time interval}

Let $u$ be the uniform law on the $2d$ directions in $\ZZ^d$. For $m\geq0$ and $\ell\geq1$, write $P_{m,\ell}$ for the law of $(X_{m+1},\ldots,X_{m+\ell})$ and put $Q_\ell:=u^{\otimes\ell}$. These are laws of the increments before projection onto the torus. For two probability measures $P,Q$ on a finite set, their relative entropy is
\[
D(P\Vert Q):=\sum_zP(z)\log\frac{P(z)}{Q(z)},
\]
with the usual convention $0\log0=0$. We use the standard entropy inequality \cite[Appendix 1, Proposition 8.2]{Kipnis1999}:
\begin{equation}
\label{relaentroine}
P(A)\leq\frac{D(P\Vert Q)+\log2}{\log(1+Q(A)^{-1})},\qquad 0<Q(A)<1.
\end{equation}
In particular, if $D(P_n\Vert Q_n)$ is bounded, then $Q_n(A_n)\to0$ implies $P_n(A_n)\to0$. We use this consequence for both cover times and ranges.

\begin{lemma}
\label{lemDPal0}
There is a constant $C=C(p,d)$ with the following properties.
\begin{enumerate}[(i)]
\item For $n\geq2$,
\begin{equation}
\label{eqfullentropy}
D(P_{0,n}\Vert Q_n)\leq C
\begin{cases}
\log n,&\alpha<1/2,\\
(\log n)^2,&\alpha=1/2,\\
n^{2\alpha-1},&\alpha>1/2.
\end{cases}
\end{equation}
\item If $\alpha<1/2$, then, for all $m,\ell\geq1$,
\begin{equation}
\label{eqlateentropy}
D(P_{m,\ell}\Vert Q_\ell)
\leq C\log\left(1+\frac{\ell}{m}\right).
\end{equation}
\item For all $1\leq\ell\leq m$,
\begin{equation}
\label{eqconditionalentropy}
D\left(\mathcal L(X_{m+1},\ldots,X_{m+\ell}\mid\FF_m)\,\Vert\,Q_\ell\right)
\leq C\left(\frac{\ell\|\Delta_m\|^2}{m^2}+\frac{\ell^2}{m^2}\right)
\quad\text{almost surely}.
\end{equation}
\end{enumerate}
\end{lemma}

\begin{proof}
Let $q_{k+1}$ denote the conditional distribution of $X_{k+1}$ given $\FF_k$. By $\log t\leq t-1$ and \eqref{PXn1},
\begin{equation}
\label{eqonestepentropy}
D(q_{k+1}\Vert u)
\leq\sum_{e}\frac{(q_{k+1}(e)-u(e))^2}{u(e)}
=\frac{2d\alpha^2}{k^2}\|\Delta_k\|^2,
\end{equation}
where we also used that $\sum_e q_{k+1}(e)-u(e) =0$. Here and below, the sum runs over the $2d$ directions $\{\pm e_1,\dots,\pm e_d\}$. Since $X_1$ has law $u$, the chain rule for relative entropy (see e.g. \cite[Theorem 2.5.3]{MR2239987}) gives
\[
D(P_{0,n}\Vert Q_n)
\leq 2d\alpha^2\sum_{k=1}^{n-1}\frac{\EE\|\Delta_k\|^2}{k^2}.
\]
Then \eqref{directioncountmoment} in Lemma~\ref{empiconvunifprop} proves (i).

Convexity of relative entropy (see e.g. \cite[Theorem 2.7.2]{MR2239987}) and \eqref{eqonestepentropy} give
\begin{equation}
\label{eqintervalentropychain}
D(P_{m,\ell}\Vert Q_\ell) \leq\EE D\left(\mathcal L(X_{m+1},\ldots,X_{m+\ell}\mid\FF_m)\,\Vert\,Q_\ell\right)\notag \leq 2d\alpha^2\sum_{k=m}^{m+\ell-1}\frac{\EE\|\Delta_k\|^2}{k^2}.
\end{equation}
If $\alpha<1/2$, the last sum is at most $C\log(1+\ell/m)$, which proves \eqref{eqlateentropy}.

For the conditional estimate, the direction-count update and \eqref{PXn1} give
\begin{equation}
\label{eqconditionalcountmoment}
\EE(\|\Delta_{k+1}\|^2\mid\FF_k)
=\left(1+\frac{2\alpha}{k}\right)\|\Delta_k\|^2+1-\frac1{2d}.
\end{equation}
To see this, write $\Delta_{k+1}-\Delta_k$ as the unit vector corresponding to $X_{k+1}$ minus the vector with all coordinates equal to $1/(2d)$. Its squared norm is $1-1/(2d)$, and its conditional mean is $\alpha\Delta_k/k$. For $k\geq m$, set
\[
c_k:=\prod_{j=m}^{k-1}\Bigl(1+\frac{2\alpha}{j}\Bigr),
\]
so that $c_m=1$ and $c_{k+1}=(1+\frac{2\alpha}{k})c_k$. Dividing \eqref{eqconditionalcountmoment} by $c_{k+1}$ and iterating, and then using $1-1/(2d)<1$, gives
\[
\EE(\|\Delta_k\|^2\mid\FF_m)
\leq c_k\|\Delta_m\|^2+\sum_{j=m+1}^{k}\frac{c_k}{c_j}.
\]
The factors are positive because $\alpha\geq-1/(2d-1)\geq-1/3$. If $\alpha<0$, then $c_k\leq1$ and $c_k/c_j\leq1$; if $\alpha\geq 0$ and $k\leq 2m$, then $c_j\geq1$, so $c_k/c_j\leq c_k\leq\exp(2\alpha\sum_{j=m}^{2m}j^{-1})\leq C$, uniformly in $m$. Thus, in either case, if $m \leq k \leq 2m$, the right-hand side is at most $C(\|\Delta_m\|^2+k-m)$. Applying the conditional chain rule and \eqref{eqonestepentropy}, one obtains (iii).
\end{proof}

\subsection{Hitting estimates and the upper tail}

We use the following standard estimates for simple random walk. A proof, including the parity issue, is given in Appendix~\ref{appsrw}.

\begin{lemma}
\label{lemSRWheatkerest}
Let $P$ be the transition matrix of simple random walk on $\ZZ_L^d$. There is a constant $C(d)>0$ such that, for all $L\geq2$ and $x,y\in\ZZ_L^d$,
\begin{equation}
\label{heatkerZldest}
P^n(x,y)\leq C(d)\left(n^{-d/2}+L^{-d}\right),\qquad n\geq1.
\end{equation}
Moreover, for every $\eta\in(0,1)$, there is $c(d,\eta)>0$ such that
\begin{equation}
\label{AnlowbdZLd}
P^n(x,y)+P^{n+1}(x,y)\geq\frac{2-\eta}{L^d},\qquad n\geq c(d,\eta)L^2.
\end{equation}
If $L$ is odd, then also $P^n(x,y)\geq(1-\eta)/L^d$ for $n\geq c(d,\eta)L^2$.
\end{lemma}

\begin{proposition}
\label{propheatkerest}
Suppose $0\leq\alpha<1$. There is a constant $C(p,d)>0$ such that, for every almost-surely finite $\FF$-stopping time $\tau$ and every $y\in\ZZ_L^d$,
\begin{equation}
\label{eqERWtoriheatk}
\PP(S_{\tau+n}=y\mid\FF_\tau)
\leq C(p,d)\left(n^{-d/2}+L^{-d}\right),\qquad n\geq1.
\end{equation}
Moreover, for every $\eta\in(0,1)$, there is $c(p,d,\eta)>0$ such that
\begin{equation}
\label{eqERWtorimix}
\PP(S_{\tau+n}=y\text{ or }S_{\tau+n+1}=y\mid\FF_\tau)
\geq\frac{2-\eta}{L^d},\qquad n\geq c(p,d,\eta)L^2.
\end{equation}
\end{proposition}

\begin{proof}
We use the conditional forest decomposition \eqref{eqconditionalforestdecomp}. After conditioning on the forest and the remainder, the displacement is a translate of simple random walk with $|\mathscr I(\tau,n)|$ steps. Thus \eqref{eqestisover} and Lemma~\ref{lemSRWheatkerest} imply
\[
\PP(S_{\tau+n}=y\mid\FF_\tau)
\leq C(d)\left(\left(\frac{8}{(1-\alpha)n}\right)^{d/2}+L^{-d}\right)
+2e^{-(1-\alpha)n/18},
\]
which proves \eqref{eqERWtoriheatk}.

For the lower bound, fix $\eta\in(0,1)$ and apply Lemma~\ref{lemSRWheatkerest} with precision $\eta/4$. Increase $c(p,d,\eta)$ so that $(1-\alpha)n/8\geq c(d,\eta/4)L^2$ and $2e^{-(1-\alpha)n/18}\leq\eta/4$. If $L$ is odd, apply the last assertion of Lemma~\ref{lemSRWheatkerest} at times $n$ and $n+1$ and add the resulting bounds. These two events are disjoint. If $L$ is even, choose $r\in\{n,n+1\}$ with the same parity as $\sum_{i=1}^d (y(i)-S_{\tau}(i))$, the displacement from $S_\tau$ to $y$. In the conditional decomposition at time $r$, the required free displacement has the same parity as the number of free steps. The other term in the left-hand side of \eqref{AnlowbdZLd} is therefore zero, so the non-zero term is at least $(2-\eta/4)/L^d$. In both cases the result is bounded from below by
\[
\frac{2-\eta/2}{L^d}\left(1-2e^{-(1-\alpha)n/18}\right)
\geq\frac{(2-\eta/2)(1-\eta/4)}{L^d}
\geq\frac{2-\eta}{L^d}.
\]
\end{proof}

We shall also use the block argument of Lemma~\ref{lemblockcover} when the conditional hitting estimate holds only on suitable events. Let $m_j=m_0+j\ell$, where $m_0\geq0$ and $\ell\geq1$ are integers, and suppose that $G_{m_j}\in\FF_{m_j}$ and $\rho\in(0,1]$ satisfy
\[
\PP\bigl(y\in\{S_{m_j},\ldots,S_{m_{j+1}}\}\mid\FF_{m_j}\bigr)\geq\rho
\quad\text{on }G_{m_j}
\]
for every $y\in\ZZ_L^d$ and $0\leq j<J$. Successively conditioning at times $m_{J-1},\ldots,m_0$ gives
\[
\PP\left(\tau_y>m_J,\ \bigcap_{j=0}^{J-1}G_{m_j}\right)
\leq(1-\rho)\PP\left(\tau_y>m_{J-1},\ \bigcap_{j=0}^{J-2}G_{m_j}\right)\leq\cdots\leq(1-\rho)^J.
\]
At each step, the current event $G_{m_j}$ is dropped after applying the conditional hitting estimate. A union bound over the vertices and the exceptional block starts therefore yields
\begin{equation}
\label{eqblockcovergood}
\PP(\tau_{\rm cov}>m_J)
\leq L^d(1-\rho)^J+\sum_{j=0}^{J-1}\PP(G_{m_j}^c).
\end{equation}
The bound also applies when $\rho$ depends on $L$.

\begin{proposition}
\label{prophighduppertail}
For every $p\in[0,1)$ and $d\geq2$, there are constants $c,C,\kappa>0$, depending only on $p,d$, such that, with
\begin{equation}
\label{defKLhighd}
K_L:=\left\lceil\frac{\kappa H(d,L)}{\log L}\right\rceil,
\end{equation}
where $H(d,L)$ is defined in \eqref{defHdLtLstar}, one has
\begin{equation}
\label{eqhighdcovertail}
\PP(\tau_{\rm cov}>n)\leq \min\{1,CL^d e^{-cn/K_L}\},\qquad n\geq CK_L.
\end{equation}
Consequently, $\EE\tau_{\rm cov}\leq C H(d,L)$, and, for each $r>0$, there is $\widetilde\kappa(p,d,r)>0$ such that
\begin{equation}
\label{uptailcovarbpow}
\PP\bigl(\tau_{\rm cov}\geq\widetilde\kappa(p,d,r)H(d,L)\bigr)
\leq C(p,d,r)L^{-r}.
\end{equation}
The family $(\tau_{\rm cov}/H(d,L))_{L\geq2}$ is uniformly integrable.
\end{proposition}

\begin{proof}
First suppose $\alpha\geq0$. Let $b_L:=\lceil c(p,d,1/2)L^2\rceil$, where $c(p,d,1/2)$ is from Proposition~\ref{propheatkerest}, and choose $\kappa$ so that $K_L\geq2b_L$. Given a finite stopping time $\tau$, define
\[
L_{\tau,y}:=\sum_{k=b_L}^{K_L}\1_{\{S_{\tau+k}=y\}}.
\]
Summing \eqref{eqERWtorimix} over adjacent pairs gives
\[
\EE(L_{\tau,y}\mid\FF_\tau)\geq\frac{K_L-b_L}{2L^d}\geq\frac{K_L}{4L^d}.
\]
On the event $L_{\tau,y}>0$, apply \eqref{eqERWtoriheatk} after the first visit to $y$ in this interval. The expected number of subsequent visits, including that first visit, is at most
\[
1+C\sum_{j=1}^{K_L}\left(j^{-d/2}+L^{-d}\right)
\leq C\frac{K_L}{L^d}.
\]
For $d=2$, the sum of $j^{-1}$ is $O(\log K_L)=O(\log L)$, which is of the same order as $K_L/L^2$. For $d\geq3$, the sum of $j^{-d/2}$ is bounded and $K_L/L^d$ is bounded away from zero. Conditioning at the first visit and comparing these two bounds therefore gives a constant $\rho=\rho(p,d)>0$ such that
\begin{equation}
\label{blockhitprobest}
\PP\bigl(y\in\{S_\tau,\ldots,S_{\tau+K_L}\}\mid\FF_\tau\bigr)\geq\rho.
\end{equation}
Lemma~\ref{lemblockcover} proves \eqref{eqhighdcovertail} in this case.

For $\alpha\leq0$, we can instead use entropy. Apply the preceding argument with $\alpha=0$ first. Repeating a fixed number of SRW blocks and increasing $\kappa$, we may arrange that simple random walk, from any starting point, hits each prescribed vertex within $K_L$ steps with probability at least $3/4$. Fix a small $\eta>0$ and, for $m\geq1$, set
\[
G_m:=\left\{\|\Delta_m\|\leq\frac{\eta m}{\sqrt{K_L}}\right\}\in\FF_m.
\]
If $m\geq j_0K_L$, Lemma~\ref{lemDPal0}(iii) bounds the conditional entropy of the next $K_L$ increments on $G_m$ by $C(\eta^2+j_0^{-2})$. Choose $\eta$ small and then $j_0$ large so that this is at most $1/8$. Pinsker's inequality (see e.g. \cite[Theorem 4.19]{boucheron2013concentration}) bounds the total variation distance by the square root of half the relative entropy. Hence, for every $y$,
\begin{equation}
\label{eqgoodblockhit}
\PP\bigl(y\in\{S_m,\ldots,S_{m+K_L}\}\mid\FF_m\bigr)\geq\frac12
\quad\text{on }G_m,\qquad m\geq j_0K_L.
\end{equation}
The concentration estimate in Lemma~\ref{empiconvunifprop} gives
\begin{equation}
\label{eqbadblockstart}
\PP(G_m^c)\leq C e^{-cm/K_L}.
\end{equation}

Take $n\geq C K_L$, put $m_0:=\lfloor n/2\rfloor$ and $m_j:=m_0+jK_L$, and let $J:=\lfloor(n-m_0)/K_L\rfloor$. Increase $C$ so that $m_0\geq j_0K_L$. Applying \eqref{eqblockcovergood} with $\ell=K_L$ and $\rho=1/2$, we obtain
\begin{align*}
\PP(\tau_{\rm cov}>n)
&\leq L^d2^{-J}+\sum_{j=0}^{J-1}\PP(G_{m_j}^c)\\
&\leq L^d2^{-J}+C\sum_{j=0}^{J-1}e^{-cm_j/K_L}
\leq CL^d e^{-cn/K_L}.
\end{align*}
The last sum is geometric, and $J\geq n/(2K_L)-1$. This proves \eqref{eqhighdcovertail} also when $\alpha\leq0$, including $p=0$.

Since $K_L\log L$ is comparable to $H(d,L)$, by \eqref{eqsummin1Nerho}, summing \eqref{eqhighdcovertail} gives the mean bound, and taking $n$ to be a sufficiently large multiple of $H(d,L)$ gives \eqref{uptailcovarbpow}. More generally, \eqref{eqhighdcovertail} implies
\[
\sup_{L\geq2}\PP\bigl(\tau_{\rm cov}>tH(d,L)\bigr)\leq C e^{-ct},\qquad t\geq C.
\]
Indeed, the factor $L^d$ is absorbed by the exponential once $t$ is large, because $H(d,L)/K_L\geq c\log L$. Integrating this last bound proves uniform integrability.
\end{proof}

\subsection{An improved upper bound in higher dimensions}

For an improved upper bound in dimensions $d\geq3$, we use blocks that are long compared with the mixing time but short compared with $L^d$. When the direction frequencies are close to uniform, the expected number of visits in such a block is at most $g_d(0)+o(1)$. The short-time return probabilities are close to those of simple random walk, and \eqref{eqERWtoriheatk} controls the remaining returns.

\begin{proposition}
\label{propsharphighdupper}
Suppose $d\geq3$ and $0\leq\alpha<1$. Then, for every $\varepsilon>0$,
\begin{equation}
\label{eqsharphighdupper}
\lim_{L\to\infty}\PP\left(\tau_{\rm cov}>(1+\varepsilon)t_L^\star\right)=0.
\end{equation}
Consequently,
\[
\limsup_{L\to\infty}\frac{\EE\tau_{\rm cov}}{t_L^\star}\leq1.
\]
\end{proposition}

\begin{proof}
Write $N=L^d$ and $g=g_d(0)$, so that $t_L^\star=gN\log N$. Fix $\varepsilon>0$, and choose $\delta,\zeta>0$ and $\eta\in(0,1)$ sufficiently small that
\begin{equation}
\label{eqsharpupperparameters}
\frac{(1+\varepsilon-\delta)(1-\eta)g}{g+\zeta}>1.
\end{equation}
Put
\[
b_L:=\lceil c(p,d,\eta)L^2\rceil,\qquad
\ell_L:=\lceil L^2\log L\rceil,\qquad
m_0:=\lceil\delta t_L^\star\rceil,
\]
where $c(p,d,\eta)$ is from \eqref{eqERWtorimix}. Notice that $b_L=o(\ell_L)$, $\ell_L=o(N)$ and $\ell_L=o(m_0)$.

For a deterministic $m\geq m_0$ and a vertex $y$, define
\[
V_{m,y}:=\sum_{k=b_L}^{\ell_L}\1_{\{S_{m+k}=y\}}.
\]
Summing \eqref{eqERWtorimix} over disjoint adjacent pairs gives, for all sufficiently large $L$,
\begin{equation}
\label{eqsharpblockvisitslower}
\EE(V_{m,y}\mid\FF_m)
\geq\frac{2-\eta}{N}\left\lfloor\frac{\ell_L-b_L+1}{2}\right\rfloor
\geq\frac{(1-\eta)\ell_L}{N}.
\end{equation}

We next bound the expected number of visits after the first one. Fix an integer $R\geq1$ and a small $\eta'>0$, to be chosen below, and set
\[
G_m:=\{\|\Delta_m\|\leq\eta'm\}\in\FF_m.
\]
Since each increment of $\Delta$ has norm at most $1$, on $G_m$ we have the deterministic bound
\begin{equation}
\label{eqsharpblockfrequencies}
\frac{\|\Delta_{m+r}\|}{m+r}
\leq\eta'+\frac{\ell_L+R}{m_0},\qquad 0\leq r\leq\ell_L+R.
\end{equation}
For fixed $R,\eta'$ and sufficiently large $L$, the right-hand side is at most $2\eta'$. By \eqref{PXn1}, the total variation distance between each next direction law and the uniform law is then at most $C(d)\eta'$. This bound also applies after any stopping time $T$ taking values in $[m,m+\ell_L]$, for the next $R$ steps. Coupling successive directions with independent uniform directions, the conditional probability of any disagreement in the first $j$ steps is at most $C(d)j\eta'$. Consequently, for $L>R$, on $G_m$,
\[
\sum_{j=0}^{R}\PP(S_{T+j}=S_T\mid\FF_T)
\leq\sum_{j=0}^{R}\PP^{(0)}(\widetilde S_j=0)+C(d)R^2\eta'.
\]
Here $L>R$ ensures that a return modulo $L$ in at most $R$ nearest-neighbor steps is a return on $\ZZ^d$. For the longer returns, we use \eqref{eqERWtoriheatk}, and consequently, on $G_m$,
\begin{equation}
\label{eqsharpblockreturns}
\begin{aligned}
\EE\left(\left.\sum_{j=0}^{\ell_L}\1_{\{S_{T+j}=S_T\}}\,\right|\FF_T\right)
&\leq g+C(d)R^2\eta'
+C(p,d)\sum_{j>R}j^{-d/2}+\frac{C(p,d)\ell_L}{N}\\
&\leq g+\zeta.
\end{aligned}
\end{equation}
To obtain the last inequality, first choose $R$ large, using $d\geq3$, then choose $\eta'$ small, and finally take $L$ large. The choices are uniform in $m\geq m_0$ and in $T$.

Apply this estimate to the bounded stopping time
\[
T:=\inf\{k\geq m+b_L:S_k=y\}\wedge(m+\ell_L).
\]
The event $A:=\{V_{m,y}>0\}$ belongs to $\FF_T$, and $S_T=y$ on $A$. On that event, all visits counted by $V_{m,y}$ occur between $T$ and $m+\ell_L$. Conditioning first on $\FF_T$ in \eqref{eqsharpblockreturns} therefore gives
\[
\EE(V_{m,y}\mid\FF_m)\leq(g+\zeta)\PP(A\mid\FF_m)
\qquad\text{on }G_m.
\]
Together with \eqref{eqsharpblockvisitslower}, this yields
\begin{equation}
\label{eqsharpblockhit}
\PP\bigl(y\in\{S_m,\ldots,S_{m+\ell_L}\}\mid\FF_m\bigr)
\geq\rho_L:=\frac{(1-\eta)\ell_L}{(g+\zeta)N}
\qquad\text{on }G_m.
\end{equation}

Set
\[
n_L:=\lfloor(1+\varepsilon)t_L^\star\rfloor,\qquad
J:=\left\lfloor\frac{n_L-m_0}{\ell_L}\right\rfloor,\qquad
m_j:=m_0+j\ell_L\quad(0\leq j\leq J).
\]
All these block starts are comparable to $m_0$, and \eqref{directioncounttail} implies
\begin{equation}
\label{eqsharpbadblocks}
\sum_{j=0}^{J-1}\PP(G_{m_j}^c)
\leq C L^{d-2}\exp\left(-\frac{c m_0^2}{V_\alpha(m_0)}\right)
\longrightarrow0.
\end{equation}
Indeed, $J\leq C L^{d-2}$ and, for each fixed $\alpha<1$, $m_0^2/V_\alpha(m_0)$ grows faster than $\log L$.
Using \eqref{eqblockcovergood} with \eqref{eqsharpblockhit}, we obtain
\[
\PP(\tau_{\rm cov}>n_L)
\leq N(1-\rho_L)^J+\sum_{j=0}^{J-1}\PP(G_{m_j}^c)
\leq N e^{-J\rho_L}+o(1).
\]
Finally,
\[
\lim_{L\to\infty}\frac{J\rho_L}{\log N}
=\frac{(1+\varepsilon-\delta)(1-\eta)g}{g+\zeta}>1
\]
by the choice \eqref{eqsharpupperparameters}, proving \eqref{eqsharphighdupper}. The assertion about the mean follows from the uniform integrability in Proposition~\ref{prophighduppertail}.
\end{proof}

\subsection{Completion of the proof}

We recall some known results for the SRW. With the normalization $t_L^\star$ defined in \eqref{defHdLtLstar}, simple random walk satisfies $\tau_{\rm cov}/t_L^\star\to1$ in probability. We also use the following stronger estimates for unusually early covering, proved in dimension two by Comets, Gallesco, Popov and Vachkovskaia~\cite[Theorem 1.1]{MR3126579} and in higher dimensions by Li, Shi and Xu~\cite[Corollary 0.3, (0.7)]{li2025ldcovertime}. For every $\gamma\in(0,1)$ and sufficiently small $\varepsilon>0$, and all sufficiently large $L$,
\begin{equation}
\label{LDPcovSRWtorilt}
\PP^{(0)}(\tau_{\rm cov}\leq\gamma t_L^\star)
\leq
\begin{cases}
\exp\bigl(-L^{2(1-\sqrt\gamma)-\varepsilon}\bigr),&d=2,\\
\exp\bigl(-L^{d(1-\gamma)-\varepsilon}\bigr),&d\geq3.
\end{cases}
\end{equation}
For the upper tail, we only use the following weaker result (\cite[Theorem 1.3]{MR3126579}, \cite[(0.7)]{li2025ldcovertime}): For any $\gamma >1$,
\begin{equation}
\label{LDPcovSRWtoriut}
\lim_{L\to \infty}\PP^{(0)}(\tau_{\rm cov}\geq \gamma t_L^\star)=0, \quad d\geq 2.
\end{equation}

For the lower bounds, the strong SRW lower-tail estimate compensates for the entropy of the whole path. In the diffusive regime, the upper bound uses a later time interval on which the entropy remains bounded. At criticality in dimensions $d\geq3$, we use Proposition~\ref{propsharphighdupper} instead.

\begin{proof}[Proof of Theorem~\ref{thmhighd}]
We first prove the lower bounds used in both parts. Recall $P_{m,\ell}$ and $Q_\ell$ defined before Lemma~\ref{lemDPal0}. If $\alpha\leq1/2$, then, for every fixed $\gamma\in(0,1)$, Lemma~\ref{lemDPal0}(i), \eqref{relaentroine} and \eqref{LDPcovSRWtorilt} give
\begin{equation}
\label{eqdiffusivecoverlower}
\PP(\tau_{\rm cov}\leq\gamma t_L^\star)\leq C L^{-c}
\end{equation}
for some $c>0$. Indeed, writing $n_L=\lceil\gamma t_L^\star\rceil$, the numerator in \eqref{relaentroine} is $D(P_{0,n_L}\Vert Q_{n_L})+\log2$, which is $O(\log L)$ if $\alpha<1/2$ and $O((\log L)^2)$ if $\alpha=1/2$. Its denominator is bounded below by a positive power of $L$.

Suppose $\alpha>1/2$. In dimension two, choose $0<\gamma<4(1-\alpha)^2$. Then choose $\varepsilon>0$ so small that
\[
2(1-\sqrt\gamma)-\varepsilon>4\alpha-2.
\]
The numerator in the entropy bound is $O(L^{4\alpha-2}(\log L)^{4\alpha-2})$, while the denominator is at least $L^{2(1-\sqrt\gamma)-\varepsilon}$. Their ratio is therefore $O(L^{-c})$ for some $c>0$. In dimensions $d\geq3$, choose $0<\gamma<2(1-\alpha)$ and then $\varepsilon>0$ so that
\[
d(1-\gamma)-\varepsilon>d(2\alpha-1).
\]
Now the numerator is $O(L^{d(2\alpha-1)}(\log L)^{2\alpha-1})$, which gives the same conclusion. Thus, for every $p<1$ and $d\geq2$, an appropriate fixed $\gamma>0$ satisfies
\[
\PP(\tau_{\rm cov}\leq\gamma t_L^\star)\leq C L^{-c}.
\]

To prove (i), first suppose that $\alpha<1/2$.
For the upper bound, fix $\varepsilon>0$ and set
\[
n_L:=\lfloor(1+\varepsilon)t_L^\star\rfloor,\qquad
m_L:=\lfloor\varepsilon t_L^\star/2\rfloor,\qquad
\ell_L:=n_L-m_L.
\]
Let $A_L$ be the event that the partial sums of a block of $\ell_L$ increments, including the initial position, do not cover $\ZZ_L^d$. Observe that $\{\tau_{\rm cov}>n_L\}$ is contained in this event for the increment block starting at time $m_L$. Indeed, if that block covers the torus, then so does the full trajectory. Moreover, translation invariance and \eqref{LDPcovSRWtoriut} imply
\[
Q_{\ell_L}(A_L)=\PP^{(0)}(\tau_{\rm cov}>\ell_L)\longrightarrow0,
\]
since $\ell_L/t_L^\star\to1+\varepsilon/2$. Lemma~\ref{lemDPal0}(ii) gives
\[
D(P_{m_L,\ell_L}\Vert Q_{\ell_L})\leq C\log\left(1+\frac{\ell_L}{m_L}\right)\leq C(p,d,\varepsilon)
\]
for all sufficiently large $L$. Applying \eqref{relaentroine} to this block event yields
\[
\PP(\tau_{\rm cov}>n_L)
\leq\frac{C(p,d,\varepsilon)+\log2}{\log(1/Q_{\ell_L}(A_L))}\longrightarrow0.
\]

If $d\geq3$ and $\alpha=1/2$, the same upper-tail conclusion follows from Proposition~\ref{propsharphighdupper}. Together with \eqref{eqdiffusivecoverlower}, this proves convergence in probability to $1$ throughout the parameter range in (i). Since $t_L^\star$ is a fixed positive multiple of $H(d,L)$, Proposition~\ref{prophighduppertail} gives uniform integrability and hence convergence in $L^1$.

For (ii), the lower bounds proved above imply $\EE\tau_{\rm cov}\geq cH(d,L)$. The upper bound for the mean and the upper-tail estimate follow from Proposition~\ref{prophighduppertail}. Combining the polynomial lower-tail bound with \eqref{uptailcovarbpow} proves \eqref{tauconcend23}.
\end{proof}

\begin{question}[Cover-time asymptotics beyond the diffusive regime]
\label{queshighdcoverlimit}
Suppose that $d=2$ and $\alpha\in[1/2,1)$, or that $d\geq3$ and $\alpha\in(1/2,1)$. Does there exist a constant $c(p,d)\in(0,\infty)$ such that
\[
\lim_{L\to\infty}\frac{\tau_{\rm cov}}{H(d,L)}
=c(p,d)\qquad\text{in }L^1?
\]
If so, does it equal the corresponding constant for simple random walk?
\end{question}

\section{Ranges in \texorpdfstring{$\ZZ^d$}{Z\textasciicircum d}}
\label{secrange}

We prove Theorem~\ref{thmrange}. In dimension two we use the interval entropy estimate from Section~\ref{secHighD}. In higher dimensions, a last-visit decomposition reduces the problem to the convergence of the direction frequencies and a summable bound on return probabilities. The latter argument applies throughout $p<1$.

\subsection{The planar diffusive regime}

Assume that $d=2$ and $p<5/8$, so that $\alpha<1/2$. For $m\geq0$ and $\ell\geq1$, write
\[
r_{m,\ell}:=|\{\widetilde S_m,\ldots,\widetilde S_{m+\ell}\}|.
\]
This quantity depends only on the increments in the interval. Let $r_\ell^{(0)}$ denote the range size of planar simple random walk through time $\ell$. We use the classical estimates of Dvoretzky and Erd\H{o}s \cite[Theorems 1 and 3]{MR0047272},
\begin{equation}
\label{eqSRWplanarrange}
\frac{\log\ell}{\ell}r_\ell^{(0)}\longrightarrow\pi
\quad\text{in probability},\qquad
\EE r_\ell^{(0)}\sim\frac{\pi\ell}{\log\ell}.
\end{equation}

\begin{proof}[Proof of Theorem~\ref{thmrange}(i)]
    Fix $\varepsilon\in(0,1)$ and set $m=\lfloor\varepsilon n\rfloor$ and $\ell=n-m$. By \eqref{eqlateentropy}, $D(P_{m,\ell}\Vert Q_\ell)$ is bounded as $n\to\infty$. Applying \eqref{relaentroine} to the exceptional events in \eqref{eqSRWplanarrange} therefore gives
\[
\frac{\log n}{n}r_{m,\ell}\longrightarrow\pi(1-\varepsilon)
\quad\text{in probability}.
\]
Since $r_{m,\ell}\leq|\mathcal R_n|$, letting $\varepsilon\downarrow0$ shows that, for every $\eta>0$,
\begin{equation}
\label{eqplanarrangelower}
\PP\left(\frac{\log n}{n}|\mathcal R_n|<\pi-\eta\right)\longrightarrow0.
\end{equation}

For the expectation upper bound, we split the path into shorter blocks of length $\ell$. We choose the initial segment to contribute $o(n/\log n)$ and the total variation error on each block to be $o(1/\log n)$, while keeping $\log\ell\sim\log n$ so that the SRW constant is unchanged. For $n$ sufficiently large, set
\[
m_0=\left\lfloor\frac{n}{(\log n)^2}\right\rfloor,
\qquad
\ell=\left\lfloor\frac{n}{(\log n)^6}\right\rfloor,
\qquad J=\left\lfloor\frac{n-m_0}{\ell}\right\rfloor.
\]
For every complete block starting at $m=m_0+j\ell$, $0\leq j<J$, Pinsker's inequality and \eqref{eqlateentropy} imply
\[
\|P_{m,\ell}-Q_\ell\|_{\rm TV}
\leq C\sqrt{\ell/m_0}\leq\frac{C}{(\log n)^2}.
\]
The range size takes values between $1$ and $\ell+1$, so the total variation bound gives
\[
\EE r_{m,\ell}
\leq\EE r_\ell^{(0)}+\frac{C\ell}{(\log n)^2}.
\]
The complete range is the union of the block ranges, the initial segment and the final incomplete block. Thus no independence between blocks is needed to obtain
\[
\EE|\mathcal R_n|
\leq m_0+\ell+2+J\left(\EE r_\ell^{(0)}
+\frac{C\ell}{(\log n)^2}\right).
\]
Here $m_0+\ell=o(n/\log n)$, $J\ell\leq n$, and $\log\ell\sim\log n$. It follows from \eqref{eqSRWplanarrange} that
\begin{equation}
\label{eqplanarrangemean}
\limsup_{n\to\infty}\frac{\log n}{n}\EE|\mathcal R_n|\leq\pi.
\end{equation}
Put $Z_n=(\log n)|\mathcal R_n|/n$. Equation~\eqref{eqplanarrangelower} and $0\leq(\pi-Z_n)_+\leq\pi$ give $\EE(\pi-Z_n)_+\to0$. Together with \eqref{eqplanarrangemean} and the identity
\[
\EE|Z_n-\pi|=\EE Z_n-\pi+2\EE(\pi-Z_n)_+,
\]
this proves Theorem~\ref{thmrange}(i). 
\end{proof}

In dimension two, these arguments can be extended to show that $\EE|\mathcal R_n|=\Theta(n/\log n)$ for every fixed $p<1$. The lower bound follows from the self-intersection estimate obtained by summing Lemma~\ref{lemrangeheatkernel}, together with Cauchy--Schwarz and Jensen's inequality. For the upper bound, one uses shorter blocks and the entropy estimate. We do not pursue these estimates here.

\begin{question}[Growth of the range at and above criticality]
For $d=2$ and $p\in[5/8,1)$, is there a constant $r(p)>0$ such that
\[
\lim_{n\to\infty}\frac{\log n}{n}|\mathcal R_n|=r(p)
\qquad\text{in probability}?
\]
If so, does the convergence also hold in $L^1$, and how does $r(p)$ depend on $p$?
\end{question}

\subsection{Higher dimensions}

We first give the last-visit argument in a form that does not require stationarity.

\begin{lemma}
\label{lemrangetransient}
Let $(Z_n)_{n\geq0}$ be an adapted process in $\ZZ^d$, with filtration $(\mathcal G_n)_{n\geq0}$ and increments in a fixed finite set. Suppose that a probability measure $\mu$ on this set satisfies
\begin{equation}
\label{eqrangeblocklaw}
\lim_{n\to\infty}\|\mathcal L(Z_{n+1}-Z_n\mid\mathcal G_n)-\mu\|_{\rm TV}
=0\qquad\text{almost surely},
\end{equation}
and that the random walk with step distribution $\mu$ is transient. If
\begin{equation}
\label{eqrangesummability}
\sum_{k\geq1}\sup_{i\geq0}\PP(Z_{i+k}=Z_i)<\infty,
\end{equation}
then
\[
\lim_{n\to\infty}\frac{|\{Z_0,\ldots,Z_n\}|}{n}=\gamma_\mu
\quad\text{in }L^1,
\]
where $\gamma_\mu$ is the probability that the random walk with step distribution $\mu$, started at the origin, never returns to the origin after time zero.
\end{lemma}

\begin{proof}
For $K\geq1$, let $q_K$ be the probability that a random walk started at zero with independent $\mu$-increments does not return to zero during its first $K$ steps. For $i\geq0$, set
\[
Y_i^{(K)}:=\1_{\{Z_{i+j}\ne Z_i,\ 1\leq j\leq K\}},
\qquad
\delta_i:=\|\mathcal L(Z_{i+1}-Z_i\mid\mathcal G_i)-\mu\|_{\rm TV}.
\]
By \eqref{eqrangeblocklaw}, $\delta_i\to0$ almost surely. Put
\[
a_i^{(K)}:=\EE(Y_i^{(K)}\mid\mathcal G_i),
\qquad D_i^{(K)}:=Y_i^{(K)}-a_i^{(K)}.
\]
Successively comparing the conditional distributions of the next $K$ increments with $\mu$ gives
\[
|a_i^{(K)}-q_K|
\leq\EE\left(\left.\sum_{r=0}^{K-1}\delta_{i+r}\,\right|\mathcal G_i\right).
\]
For completeness, one may couple each next increment with a fresh $\mu$-increment until the first disagreement. The conditional probability of a disagreement at a given step, while the earlier increments agree, is at most the corresponding $\delta_{i+r}$; summing these probabilities gives the bound. Since $\delta_i\leq1$, bounded convergence and the tower property give
\[
\EE|a_i^{(K)}-q_K|
\leq\sum_{r=0}^{K-1}\EE\delta_{i+r}\longrightarrow0
\qquad(i\to\infty).
\]

Moreover, $D_i^{(K)}$ is $\mathcal G_{i+K}$-measurable and $\EE(D_i^{(K)}\mid\mathcal G_i)=0$. If $j\geq i+K$, then
\[
\EE(D_i^{(K)}D_j^{(K)})
=\EE\bigl[D_i^{(K)}\EE(D_j^{(K)}\mid\mathcal G_j)\bigr]=0.
\]
Thus $\EE(D_i^{(K)}D_j^{(K)})=0$ whenever $|i-j|\geq K$. Since $|D_i^{(K)}|\leq1$, for $n\geq K$,
\[
\EE\left|\frac1n\sum_{i=0}^{n-K}D_i^{(K)}\right|^2
\leq\frac{2K-1}{n}.
\]
Combining this with Ces\`aro averaging of $a_i^{(K)}$ gives
\begin{equation}
\label{eqrangetruncatedlimit}
\frac1n\sum_{i=0}^{n-K}Y_i^{(K)}\longrightarrow q_K
\quad\text{in }L^1.
\end{equation}

Count each visited site at its last visit before time $n$:
\[
|\{Z_0,\ldots,Z_n\}|
=\sum_{i=0}^n\1_{\{Z_{i+j}\ne Z_i,\ 1\leq j\leq n-i\}}.
\]
Comparison with the truncated sum and a union bound give
\[
\EE\left|
\frac{|\{Z_0,\ldots,Z_n\}|}{n}
-\frac1n\sum_{i=0}^{n-K}Y_i^{(K)}
\right|
\leq\frac Kn+\sum_{k>K}\sup_{i\geq0}\PP(Z_{i+k}=Z_i).
\]
Indeed, the last $K$ indices contribute at most $K$, and an earlier index counted by $Y_i^{(K)}$ but not as a last visit must have a return at some lag $k>K$. First let $n\to\infty$ and use \eqref{eqrangetruncatedlimit}. Then let $K\to\infty$, using \eqref{eqrangesummability} and $q_K\downarrow\gamma_\mu$.
\end{proof}

The following estimate supplies \eqref{eqrangesummability} for the ERW. We retain independent directions at leaves of the recursive parent tree, which works for every $p<1$.

\begin{lemma}
\label{lemrangeheatkernel}
For every $d\geq2$ and $p\in[0,1)$, there exists $C=C(d,p)$ such that, for all $m\geq0$ and $k\geq1$,
\begin{equation}
\label{eqrangeheatkernel}
\sup_{x\in\ZZ^d}
\PP(\widetilde S_{m+k}-\widetilde S_m=x\mid\FF_m)
\leq Ck^{-d/2}\qquad\text{almost surely}.
\end{equation}
\end{lemma}

\begin{proof}
We construct the ERW on a recursive parent tree. The parent choices $u_j$, $j\geq2$, are independent, with $u_j$ uniform on $[j-1]$, and each $j$ is joined to $u_j$. We retain every parent edge, whereas the forest in Section~\ref{secnotation} retains only some of them. A vertex is a leaf at time $n$ if it has no children among $1,\ldots,n$.

Independently of the parents, take i.i.d.\ marks $\zeta_j$, $j\geq1$, uniform on $[0,1]$. Fix an ordering of the $2d$ directions. The $2d$ equal subintervals of $[0,1]$ determine $X_1$ from $\zeta_1$ in this order. For $j\geq2$, given $X_{u_j}=s$, set $X_j=s$ if $\zeta_j\leq p$. Otherwise, divide $[0,1]$ into $2d-1$ equal subintervals, and let $X_j$ be the direction corresponding to the subinterval containing $(\zeta_j-p)/(1-p)$, using the fixed order of the remaining directions. Values at interval endpoints may be assigned arbitrarily. Thus, given its parent direction $s$, a vertex has direction law
\[
\nu_s=p\delta_s+\frac{1-p}{2d-1}
\sum_{e\in\{\pm e_1,\ldots,\pm e_d\}\setminus\{s\}}\delta_e.
\]
Here $\delta_e$ is the Dirac measure at direction $e$. For $d\geq2$ and $p<1$, this law is not supported on an affine hyperplane. This remains true at $p=0$. Consequently, Esseen's concentration bound \cite[Theorem 6.2 and its corollary]{MR231419} applied with a ball of radius less than $1/2$, gives
\begin{equation}
\label{eqleafconvolution}
\sup_x\nu_s^{*h}(x)\leq C(d,p)h^{-d/2},\qquad h\geq1.
\end{equation}
The constant is uniform over the finitely many directions $s$.

We next bound the number of leaves among $m+1,\ldots,m+k$ in the tree through time $m+k$. It suffices to consider $k\geq2$. For $m\geq1$, put
\[
\widehat{\FF}_m:=\sigma(u_2,\ldots,u_m,\zeta_1,\ldots,\zeta_m)
\supseteq\FF_m,
\]
and let $\widehat{\FF}_0$ be the trivial $\sigma$-field. We first prove the estimate conditionally on $\widehat{\FF}_m$. Independently add variables $w_j$, uniform on $[j-1]$, for $2\leq j\leq k$. Define comparison parents by
\[
\widehat u_j:=
\begin{cases}
u_{m+j}-m,&u_{m+j}>m,\\
w_j,&u_{m+j}\leq m.
\end{cases}
\]
These variables are independent of $\widehat{\FF}_m$ and of one another. For each $\ell \in[j-1]$,
\[
\PP(\widehat u_j=\ell )
=\frac1{m+j-1}+\frac{m}{m+j-1}\frac1{j-1}
=\frac1{j-1}.
\]
Thus they define a fresh recursive tree on $[k]$. Every leaf $j$ of this tree corresponds to a leaf $m+j$ of the original tree: a later vertex with parent $m+j$ would have comparison parent $j$. The comparison root is not a leaf when $k\geq2$.

Let $V_k$ be the number of comparison leaves. For $2\leq j\leq k$, its probability of being a leaf is
\[
\prod_{h=j+1}^k\left(1-\frac1{h-1}\right)=\frac{j-1}{k-1},
\]
so $\EE V_k=k/2$. Also, $V_k$ equals $k$ minus the number of distinct parent labels. Changing one parent therefore changes $V_k$ by at most one. McDiarmid's inequality yields
\begin{equation}
\label{eqrangeleaftail}
\PP(V_k<k/3)\leq
\exp\left(-\frac{k^2}{18(k-1)}\right)
\leq e^{-k/18}.
\end{equation}
This bound also holds conditionally on $\widehat{\FF}_m$. 

Keeping $\widehat{\FF}_m$ revealed, now reveal all parent choices through time $m+k$ and the marks at the non-leaf vertices of the future block. All directions outside the future leaves, and the parent directions of those leaves, are then known. The unrevealed marks remain independent. Partition the future leaves according to their parent directions and choose a largest class (there might be more than one). On the event that there are at least $k/3$ future leaves, this class has size $h\geq k/(6d)$. Reveal also the marks at all leaves outside the chosen class. The increment sum over the block is now a fixed vector plus a sum of $h$ independent variables with a common law $\nu_s$ if the parent direction is $s$. The class was chosen without using its marks, and no leaf direction enters any other increment, so this conditional independence is preserved.

Apply \eqref{eqleafconvolution} on this event and bound the conditional point probabilities by one on its complement. Since the number of future leaves is at least $V_k$, \eqref{eqrangeleaftail} and the tower property give
\[
\sup_x\PP(\widetilde S_{m+k}-\widetilde S_m=x\mid\widehat{\FF}_m)
\leq Ck^{-d/2}+e^{-k/18}.
\]
Finally, for every $x\in\ZZ^d$,
\begin{align*}
\PP(\widetilde S_{m+k}-\widetilde S_m=x\mid\FF_m)
&=\EE\bigl[\PP(\widetilde S_{m+k}-\widetilde S_m=x\mid\widehat{\FF}_m)
\mid\FF_m\bigr]\\
&\leq Ck^{-d/2}+e^{-k/18}.
\end{align*}
Absorbing the exponential term and the case $k=1$ into the constant proves the lemma. 
\end{proof}

\begin{proof}[Proof of Theorem~\ref{thmrange}(ii)]
By \eqref{directioncounttail}, for every $\varepsilon>0$,
\[
\sum_{n\geq1}\PP(\|\Delta_n\|>\varepsilon n)<\infty,
\]
because $\alpha<1$. Hence $\Delta_n/n\to0$ almost surely, and \eqref{PXn1} proves \eqref{eqrangeblocklaw} for the uniform law on the $2d$ directions. If $d\geq3$, \eqref{eqrangeheatkernel} gives \eqref{eqrangesummability}. Lemma~\ref{lemrangetransient} now yields the result.
\end{proof}

\section{Cover-time bounds on finite groups}
\label{secfinitegroup}

The conditional hitting argument also gives cover-time bounds for more general processes. We first state a criterion in terms of conditional mixing, and then apply it to generalized step-reinforced random walks.

\subsection{A conditional mixing-to-cover estimate}

Let $Z=(Z_n)_{n\geq0}$ be a process on a finite set $E$ with $N:=|E|>1$, adapted to a filtration $(\mathcal A_n)_{n\geq0}$. For $n\geq0$, define
\begin{equation}
\label{defconditionalD}
D_Z(n):=\sup_{m\geq0}\operatorname*{ess\,sup}\max_{y\in E}
\left|N\PP(Z_{m+n}=y\mid\mathcal A_m)-1\right|.
\end{equation}
Since $Z_m$ is $\mathcal A_m$-measurable, $D_Z(0)=N-1$. The tower identity
\[
\PP(Z_{m+n+1}=y\mid\mathcal A_m)
=\EE\bigl[\PP(Z_{m+n+1}=y\mid\mathcal A_{m+1})\mid\mathcal A_m\bigr]
\]
gives $D_Z(n+1)\leq D_Z(n)$. Moreover, the bound in \eqref{defconditionalD} holds with $m$ replaced by any almost surely finite stopping time $\sigma$, by conditioning on the events $\{\sigma=m\}$.   

The argument in the proof of Proposition  \ref{prophighduppertail} can be used to prove the following. The main difference is that to upper bound $\EE(L_{\tau,y}\mid\FF_\tau)$, we shall use $D_Z$ instead of the heat kernel estimates \eqref{eqERWtoriheatk}.

\begin{proposition}
\label{propmixingcovergeneral}
Assume that
\[
T:=\inf\{n\geq1:D_Z(n)\leq1/2\}<\infty,
\qquad
A:=\sum_{n=0}^T D_Z(n).
\]
There is an absolute constant $C>0$ such that
\begin{equation}
\label{generalhitcovermean}
\max_{y\in E}\EE\tau_y\leq CA,
\qquad
\EE\tau_{\rm cov}\leq CA\log N,
\end{equation}
where $\tau_y:=\inf\{n\geq0:Z_n=y\}$ and $\tau_{\rm cov}:=\max_{y\in E}\tau_y$. Moreover, for every $r\geq0$,
\begin{equation}
\label{generalcoverexptail}
\PP\left(\tau_{\rm cov}>CA(\log N+r)\right)\leq e^{-r}.
\end{equation}
\end{proposition}

\begin{proof}
Fix $m\geq0$ and $y\in E$, and set
\[
L_{m,y}:=\sum_{n=T}^{2T-1}\1_{\{Z_{m+n}=y\}}.
\]
Since $D_Z(n)\leq1/2$ for $n\geq T$,
\[
\EE(L_{m,y}\mid\mathcal A_m)\geq\frac{T}{2N}.
\]
Conditioning at the first visit to $y$ in this interval and using the stopping-time form of \eqref{defconditionalD}, we also have
\begin{align*}
\EE(L_{m,y}\mid\mathcal A_m)
&\leq\PP(L_{m,y}>0\mid\mathcal A_m)
\left(1+\frac1N\sum_{n=1}^T(1+D_Z(n))\right)\\
&\leq\frac{CA}{N}\PP(L_{m,y}>0\mid\mathcal A_m).
\end{align*}
Indeed, the first visit contributes $1$, and at most $T$ further times remain. For the last inequality, note that $D_Z(0)=N-1$ and $D_Z(n)>1/2$ for $1\leq n<T$, so $A\geq c(N+T)$. Comparing the two bounds shows that
\[
\PP\bigl(Z_{m+n}=y\text{ for some }1\leq n\leq2T
\mid\mathcal A_m\bigr)\geq\frac{cT}{A}.
\]
Lemma~\ref{lemblockcover}, with block length $2T$, proves both conclusions.
\end{proof}

\subsection{Generalized step-reinforced random walks}

Let $G$ be a finite group with identity $e_G$ and let $\mu$ be a probability measure on $G$. We use the generalized model of \cite[Definition 1 and Remark 1.1]{peres2026transition}, including transformations which depend on the past forest. Take mutually independent sequences $(g_n)_{n\geq1}$, $(\xi_n)_{n\geq2}$ and $(u_n)_{n\geq2}$, where the $g_n$ are i.i.d. with law $\mu$, the $\xi_n$ are i.i.d. Bernoulli with parameter $\widetilde\alpha\in[0,1)$, and the $u_n$ are independent and uniform on $\{1,\ldots,n-1\}$. Let $\Theta$ be independent of these three sequences, and let
\[
\mathcal T_n(g)=\mathscr T_n\bigl(\Theta,(\xi_j,u_j)_{2\leq j<n},g\bigr)
\]
for measurable maps $\mathscr T_n$. Set $X_1^\star=g_1$ and
\[
X_n^\star=\begin{cases}
g_n,&\xi_n=0,\\
\mathcal T_n(X_{u_n}^\star),&\xi_n=1,
\end{cases}
\qquad
S_0^\star=e_G,\qquad S_n^\star=X_1^\star\cdots X_n^\star.
\]
Let $\F_n^\star$ be the forest on $[n]$ with an edge between $j$ and $u_j$ exactly when $\xi_j=1$. A vertex is isolated if it has no incident edges in this forest. The transformations need not be independent of one another. They may depend on the past reinforcement variables, but are independent of the fresh labels $(g_n)$. Given all pairs $(\xi_j,u_j)$, $\Theta$ and the non-isolated labels up to a fixed time, the isolated labels therefore remain independent with law $\mu$.

Write $N:=|G|>1$ and $P_\mu(x,y):=\mu(x^{-1}y)$. Assume that $P_\mu$ is irreducible and aperiodic. Let $\FF_m^\star:=\sigma(S_0^\star,\ldots,S_m^\star)$ and let $D(n):=D_{S^\star}(n)$, as in \eqref{defconditionalD}, for this filtration. Put
\[
d_\mu(k):=\max_{x,y\in G}|NP_\mu^k(x,y)-1|.
\]
We consider the following conditions:
\begin{enumerate}[(I)]
\item $\mu$ is a class function, that is, $\mu(xy)=\mu(yx)$ for all $x,y\in G$.
\item $\mu$ is symmetric.
\item $\mu(e_G)\geq\mu_0$ for some $\mu_0\in(0,1/2]$.
\end{enumerate}
Under Condition (II), let $\lambda_*$ be the largest absolute value of an eigenvalue of $P_\mu$ other than $1$. Under Condition (III), define
\begin{equation}
\label{defphimu}
\phi_\mu(u):=\inf_{0<|B|\leq uN}
\frac{\sum_{x\in B,y\in B^c}P_\mu(x,y)}{|B|},
\qquad \frac1N\leq u\leq\frac12,
\end{equation}
and set $\phi_\mu(u)=\phi_\mu(1/2)$ for $u>1/2$.

\begin{proposition}
\label{propconditionalmixgroup}
Let $q_n:=\lfloor(1-\widetilde\alpha)n/8\rfloor$. The following estimates hold for every $n\geq1$.
\begin{enumerate}[(i)]
\item Under Condition (I),
\[
D(n)\leq d_\mu(q_n)+2N e^{-(1-\widetilde\alpha)n/18}.
\]
\item Under Condition (II),
\[
D(n)\leq N\lambda_*^{q_n}+2N e^{-(1-\widetilde\alpha)n/18},
\]
with the convention $0^0=1$.
\item Under Condition (III), there exists $C(\mu_0)>0$ such that, for every $\varepsilon\in(0,N-1]$,
\begin{equation}
\label{eqDnisoprofile}
D(n)\leq\varepsilon
\quad\text{whenever}\quad
n\geq\frac{C(\mu_0)}{1-\widetilde\alpha}
\int_{4/N}^{8/\varepsilon}\frac{du}{u\phi_\mu(u)^2}.
\end{equation}
For $\varepsilon>N-1$, the bound follows from $D(n)\leq N-1$.
\end{enumerate}
\end{proposition}

\begin{proof}
Fix a time interval $\{m+1,\ldots,m+n\}$. Here isolation is determined in the whole forest $\F_{m+n}^\star$. First reveal $\Theta$, all pairs $(\xi_j,u_j)_{2\leq j\leq m+n}$, all labels $g_j$ for $j\leq m$, and the labels at the non-isolated vertices of the block. Revealing the parent choices even at deleted edges makes all transformations up to time $m+n$ known. For each isolated vertex $j$ of $\F_{m+n}^\star$, we have
$X_j^\star=g_j$, and no other increment up to time $m+n$
depends on $g_j$. The increments at the non-isolated vertices
of the block are therefore determined by the revealed
information. Denote their values by $h_j$.
The unrevealed labels at the isolated vertices remain
independent with law $\mu$. In other words, under this conditioning the transition kernels are
\[
K_j(x,y)=
\begin{cases}
\mu(x^{-1}y),&j\text{ is isolated in }\F_{m+n}^\star,\\
\1_{\{y=xh_j\}},&j\text{ is not isolated},
\end{cases}
\qquad m<j\leq m+n.
\]
The conditional endpoint law is therefore $(K_{m+1}\cdots K_{m+n})(S_m^\star,\cdot)$, as in \cite[Proposition 2.3]{peres2026transition}. We prove the kernel estimates under this conditioning and then take conditional expectations given $\FF_m^\star$. By \cite[Lemma 2.2]{peres2026transition}, the number $J_{m,n}$ of isolated vertices in the block satisfies
\[
\PP\left(J_{m,n}\leq\frac{(1-\widetilde\alpha)n}{8}
\,\middle|\,\FF_m^\star\right)
\leq2e^{-(1-\widetilde\alpha)n/18}.
\]

Under Condition (I), $P_\mu$ commutes with every right translation, so the relative $L^\infty$ distance on the complementary event is at most $d_\mu(q_n)$. Under Condition (II), right translations are isometries on the mean-zero subspace of $L^2$ of the uniform measure, while $P_\mu$ has norm $\lambda_*$ there. The product thus has norm at most $\lambda_*^{q_n}$, which gives the pointwise bound $N\lambda_*^{q_n}$ by Cauchy--Schwarz. In either case the exceptional event contributes at most $2N e^{-(1-\widetilde\alpha)n/18}$.

For (iii), put $J(\varepsilon):=\int_{4/N}^{8/\varepsilon}du/(u\phi_\mu(u)^2)$. Since $P_\mu$ is doubly stochastic, for every $B\subset G$,
\[
\sum_{x\in B,y\in B^c}P_\mu(x,y)
=\sum_{x\in B^c,y\in B}P_\mu(x,y).
\]
Thus the reversed kernel has the same conductance profile. Moreover, $e_G\in\operatorname{supp}\mu$, so irreducibility implies that $\operatorname{supp}\mu\,(\operatorname{supp}\mu)^{-1}$ generates $G$. The deterministic evolving-set estimate in \cite[Proposition 4.5 and the proof of Proposition 1.9]{Peres2026mix}, applied to the forward and reversed products, therefore bounds the relative $L^\infty$ distance of the conditional law by $\varepsilon/2$ whenever there are at least $C_0(\mu_0)J(\varepsilon)$ factors $P_\mu$. The estimate is uniform over the translation factors, so it applies to the conditional product above. Choosing $C(\mu_0)$ sufficiently large in \eqref{eqDnisoprofile} ensures that $q_n\geq C_0(\mu_0)J(\varepsilon)$. Finally, $\phi_\mu\leq1$ gives
\[
J(\varepsilon)\geq\log\frac{2N}{\varepsilon}.
\]
A further increase of $C(\mu_0)$ makes the exceptional contribution at most $\varepsilon/2$, proving (iii).
\end{proof}

\begin{corollary}
\label{corlazycover}
Suppose that $P_\mu$ is the lazy simple random walk kernel on $\ZZ_L^d$, $d\geq2$, and let $S^\star$ be a generalized step-reinforced random walk as above. There exists $C=C(\widetilde\alpha,d)>0$ such that
\[
\EE\tau_{\rm cov}\leq C H(d,L),
\qquad
\PP\left(\tau_{\rm cov}>C H(d,L)\right)\leq L^{-d},
\]
where $H(d,L)$ is the scale defined before Theorem~\ref{thmhighd}.
\end{corollary}

\begin{proof}
The torus is abelian, so Proposition~\ref{propconditionalmixgroup}(i) applies. For the lazy simple random walk, Fourier inversion and $1-\cos u\geq2u^2/\pi^2$ for $|u|\leq\pi$ give
\[
d_\mu(k)\leq
\left(1+2\sum_{j=1}^\infty e^{-ckj^2/L^2}\right)^d-1,
\qquad k\geq1.
\]
For $k\leq L^2$, comparison with a Gaussian integral bounds the right-hand side by $CL^d k^{-d/2}$, see e.g. Lemma \ref{lemgaussum} below. For $k\geq L^2$, summing the exponentially decreasing terms bounds it by $Ce^{-ck/L^2}$. Combining these estimates with Proposition~\ref{propconditionalmixgroup}(i), and adjusting constants depending on $\widetilde\alpha,d$, yields
\[
D(n)\leq\begin{cases}
CL^d(n+1)^{-d/2},&0\leq n\leq L^2,\\
Ce^{-cn/L^2},&n\geq L^2.
\end{cases}
\]
Here the term $2L^d e^{-(1-\widetilde\alpha)n/18}$ is absorbed into the displayed bounds. With $T$ and $A$ as in Proposition \ref{propmixingcovergeneral}, it follows that $T\leq CL^2$ and
\[
A=\sum_{n=0}^T D(n)\leq
\begin{cases}
CL^2\log L,&d=2,\\
CL^d,&d\geq3.
\end{cases}
\]
The difference between the two cases comes from summing $(n+1)^{-d/2}$ up to $L^2$. Proposition~\ref{propmixingcovergeneral}, with $r=\log N$ in its tail bound, proves the result.
\end{proof}

\appendix
\section{Simple random walk estimates on the torus}
\label{appsrw}

\begin{lemma}
\label{lemgaussum}
There is an absolute constant $C>0$ such that, for every $u\in[-L/2,L/2]$ and $\kappa>0$,
\[
\sum_{z\in\ZZ}e^{-\kappa(u+Lz)^2}
\leq C\left(1+\frac1{L\sqrt\kappa}\right).
\]
\end{lemma}

\begin{proof}
For $z\geq1$, $|u+Lz|\geq L(z-1/2)$, and the same bound holds for $z\leq-1$ after replacing $z$ by $|z|$. Hence
\[
\sum_{z\in\ZZ}e^{-\kappa(u+Lz)^2}
\leq1+2\sum_{z=1}^\infty e^{-\kappa L^2(z-1/2)^2}.
\]
Since $x\mapsto e^{-\kappa L^2x^2}$ is decreasing on $[0,\infty)$,
\[
\sum_{z=1}^\infty e^{-\kappa L^2(z-1/2)^2}
\leq2\int_0^\infty e^{-\kappa L^2x^2}\,dx
=\frac{\sqrt\pi}{L\sqrt\kappa}.
\]
This proves the claim.
\end{proof}

\begin{proof}[Proof of Lemma \ref{lemSRWheatkerest}]
Let $\widetilde P$ be the transition kernel of simple random walk on $\ZZ^d$. The standard Gaussian upper bound for finite-range symmetric random walks, see Hebisch and Saloff-Coste \cite{Hebisch1993}, gives
\[
\widetilde P^n(\widetilde x,\widetilde y)
\leq \frac{C(d)}{n^{d/2}}
\exp\left(-\frac{c(d)\|\widetilde x-\widetilde y\|^2}{n}\right),
\qquad \widetilde x,\widetilde y\in\ZZ^d.
\]
Choose $u\in[-L/2,L/2]^d\cap\ZZ^d$ representing $y-x$ modulo $L$. Periodizing the preceding estimate and applying Lemma \ref{lemgaussum} coordinatewise, we obtain
\begin{align*}
P^n(x,y)
&\leq \frac{C(d)}{n^{d/2}}
\sum_{z\in\ZZ^d}\exp\left(-\frac{c(d)\|u+Lz\|^2}{n}\right)\\
&\leq \frac{C(d)}{n^{d/2}}
\left(1+\frac{\sqrt n}{L}\right)^d
\leq C(d)\left(n^{-d/2}+L^{-d}\right),
\end{align*}
which proves \eqref{heatkerZldest}.

It remains to prove the lower bound. Fix $\eta\in(0,1)$. Suppose first that $L$ is even and set
\[
A^n(x,y):=\frac{P^n(x,y)+P^{n+1}(x,y)}2,
\qquad
 d^{(\infty)}(n):=\max_{x,y\in\ZZ_L^d}|L^dA^n(x,y)-1|.
\]
By translation invariance it is enough to take $x=0$. Fourier inversion gives
\begin{equation}
\label{FourierinvA}
L^dA^n(0,y)-1
=\sum_{\ell\in\ZZ_L^d\setminus\{0\}}
\frac{1+\lambda_\ell}{2}\lambda_\ell^n\chi_\ell(-y),
\end{equation}
where
\[
\chi_\ell(y):=\exp(2\pi\mathrm i\,\ell\cdot y/L),
\qquad
\lambda_\ell:=\frac1d\sum_{j=1}^d
\cos\left(\frac{2\pi\ell_j}{L}\right).
\]
Let $w=(L/2,\ldots,L/2)$. Since $\lambda_{\ell+w}=-\lambda_\ell$, pairing $\ell$ and $\ell+w$ in \eqref{FourierinvA} yields
\begin{align*}
d^{(\infty)}(n)
&\leq\sum_{\substack{\ell\in\ZZ_L^d\setminus\{0\}\\ \lambda_\ell\geq0}}\lambda_\ell^n
\leq\sum_{\ell\in\ZZ_L^d\setminus\{0\}}e^{-n(1-\lambda_\ell)}\\
&=\left[\sum_{k=0}^{L-1}
\exp\left(-\frac nd\left(1-\cos\frac{2\pi k}{L}\right)\right)\right]^d-1.
\end{align*}
Using $1-\cos u\geq2u^2/\pi^2$ for $|u|\leq\pi$, we get
\[
d^{(\infty)}(n)
\leq\left(1+2\sum_{k=1}^{\infty}
 e^{-8nk^2/(dL^2)}\right)^d-1.
\]
The Gaussian integral bound shows that the right-hand side is at most $\eta/2$ whenever $n\geq c(d,\eta)L^2$, after increasing $c(d,\eta)$. Hence
\[
A^n(x,y)\geq\frac{1-\eta/2}{L^d},
\]
which is equivalent to \eqref{AnlowbdZLd}.

When $L$ is odd, Fourier inversion gives
\[
\max_{y\in\ZZ_L^d}|L^dP^n(0,y)-1|
\leq\sum_{\ell\in\ZZ_L^d\setminus\{0\}}|\lambda_\ell|^n.
\]
Represent each coordinate of $\ell$ by an integer in
$[-(L-1)/2,(L-1)/2]$. If $\lambda_\ell\geq0$, then
\[
1-|\lambda_\ell|=1-\lambda_\ell
\geq\frac{c(d)}{L^2}\sum_{j=1}^d\ell_j^2.
\]
If $\lambda_\ell<0$, put $r_j:=L/2-|\ell_j|\in[1/2,L/2]$. Then
\[
1-|\lambda_\ell|=1+\lambda_\ell
\geq\frac{c(d)}{L^2}\sum_{j=1}^d r_j^2.
\]
Consequently, the contribution of the non-negative eigenvalues is bounded by a $d$-dimensional Gaussian sum over $\ZZ^d\setminus\{0\}$, while the contribution of the negative eigenvalues is bounded by a Gaussian sum over the half-integer lattice $(\ZZ+1/2)^d$. Both sums can be made smaller than $\eta/4$ by taking $n\geq c(d,\eta)L^2$ with $c(d,\eta)$ sufficiently large. Hence
\[
\max_{x,y\in\ZZ_L^d}|L^dP^n(x,y)-1|\leq\frac\eta2,
\]
and therefore $P^n(x,y)\geq(1-\eta/2)/L^d$. This proves both \eqref{AnlowbdZLd} and the claimed one-step lower bound when $L$ is odd.
\end{proof}

\section{Acknowledgments}

The author would like to thank Yuval Peres for helpful comments on a first version of the manuscript. Shuo Qin is supported by the China Postdoctoral Science Foundation under Grant Numbers 2025M773086 and 2026T190815, and National Natural Science Foundation of China under Grant Number 12271284. 

\bibliographystyle{plain}
\bibliography{math_ref}

\end{document}